\documentclass{article}
\usepackage{graphicx} 
\usepackage{amsmath,amsfonts,amssymb,amsthm,thmtools,hyperref,cleveref,xcolor}
\theoremstyle{definition}
\newtheorem{theorem}{Theorem}
\newtheorem{lemma}{Lemma}
\newtheorem{cor}{Corollary}
\newtheorem{conjecture}{Conjecture}
\newtheorem{definition}{Definition}
\newtheorem{prop}{Proposition}
\newcommand{\Calg}{\mathcal{C}_{\text{alg}}}
\newcommand{\Cab}{\mathcal{C}_{\text{rad}}}
\newtheorem{example}{Example}
\newtheorem{remark}{Remark}
\DeclareMathOperator{\botn}{gw}
\DeclareMathOperator{\Deck}{Deck}
\DeclareMathOperator{\abot}{bot_{\text{rad}}}
\newcommand{\ZZ}{\mathbb{Z}}
\newcommand{\QQ}{\mathbb{Q}}
\newcommand{\VV}{\mathbb{V}}
\newcommand{\PP}{\mathbb{P}}
\newcommand{\CC}{\mathbb{C}}
\newcommand{\RR}{\mathbb{R}}
\newcommand{\QQalg}{\overline{\QQ}}
\newcommand{\FF}{\mathbb{F}}

\newcommand{\KK}{\mathbb{K}}
\newcommand{\LL}{\mathbb{L}}
\DeclareMathOperator{\Gal}{Gal}
\DeclareMathOperator{\ED}{ED}
\DeclareMathOperator{\ML}{ML}
\Crefname{conjecture}{Conjecture}{Conjectures}
\newcommand{\gr}[1]{\textcolor{black}{#1}}
\newcommand{\bl}[1]{\textcolor{black}{#1}}
\newcommand{\gl}[1]{\textcolor{black}{#1}}
\makeatletter
\newcommand{\address}[1]{\gdef\@address{#1}}
\newcommand{\email}[1]{\gdef\@email{\url{#1}}}
\newcommand{\@endstuff}{\par\vspace{\baselineskip}\noindent\small
\begin{tabular}{@{}l}\scshape\@address\\\textit{E-mail address:} \@email\end{tabular}}
\AtEndDocument{\@endstuff}
\makeatother
\title{
A Galois-Theoretic Complexity Measure for Solving Systems of Algebraic Equations}
\author{Timothy Duff}
\address{Department of Mathematics, University of Missouri, Columbia MO, 65211}
 \email{tduff@missouri.edu}

\begin{document}

\maketitle

\abstract{
Motivated by applications of algebraic geometry, we introduce the \emph{Galois width}, a quantity characterizing the complexity of solving algebraic equations in a restricted model of computation allowing only field arithmetic and adjoining polynomial roots.
We explain why practical heuristics such as monodromy give (at least) lower bounds on this quantity, and discuss problems in geometry, optimization, statistics, and computer vision for which knowledge of the Galois width either leads to improvements over standard solution techniques or rules out this possibility entirely.
}

\tableofcontents

\section{Introduction}\label{sec:intro}

Galois theory, since its inception, has provided a powerful toolkit for answering questions of computability and complexity in algebra.
Building on the classical impossibility theorems, more recent works span a diverse array of applications, including complex dynamics~\cite{mcmullen}, computer vision~\cite{hartley}, enumerative geometry~\cite{farb}, graph drawing~\cite{bannister}, and optimization~\cite{bajaj}.

In this paper, our main motivation is the increasingly important role of
Galois theory in applications of algebraic geometry---especially for problems where the goal is to compute solutions of a system of algebraic equations.
In practice, these equations are multivariate polynomials or rational functions with coefficients in $\QQ,$ and the number of complex solutions (the \emph{degree} of the system) has often been proposed as a measure of a problem's complexity.
\Cref{sec:examples} recalls such problems in detail.

The purpose of this paper is to introduce a Galois-theoretic refinement of the degree that better reflects the practical complexity of solving.
In contrast to existing theories, which focus on upper bounds for the complexity of solving, our framework naturally facilitates complexity \emph{lower bounds} on the complexity of solving in a simplified, yet plausible, algebraic model of computation.
At the same time, for some highly-structured problems, the Galois width is smaller than \emph{a priori} upper bounds suggest, which can be exploited to develop more efficient algorithms.

\Cref{sec:bottleneck} contains the definitions and results at the core of this work. We define the \emph{Galois width} $\botn (\alpha)$ of an algebraic number $\alpha \in \QQalg$, a minimax measure of how much a solution costs in our computational model.
A similar cost $\botn (G)$ may be defined for any finite group $G$.
\Cref{thm:complexity} states that the Galois width of $\alpha $ and the Galois group of its minimal polynomial coincide.
The latter can be computed explicitly using composition series and minimal permutation degrees of simple groups (\Cref{thm:compute-bot}.)
 
Before turning to illustrate our theory on examples, we first describe in~\Cref{subsec:br-thin} a general setup to which they all adhere.
We propose to study $\botn (\bullet)$ for concrete systems of equations via Galois group computation, including heuristics based in both exact computation (eg.~Frobenius elements) and 
inexact numerical methods (eg.~monodromy.)
For the latter,~\Cref{cor:hit} uses Hilbert's irreducibility theorem~\cite[Ch.~3]{serre} to precisely state the following principle: \emph{the Galois/monodromy group of a parametrized system of algebraic equations provides a lower bound on the Galois width of its solutions for a dense subset of rational parameters.}

We note two limitations of the aforementioned Galois group heuristics: (1) they might only identify a subgroup of the full Galois group, and (2) the numerical methods are not certified.
Nevertheless, heuristics offer insight into examples that might otherwise be inaccessible. 
On the other hand, as long as our goal is to obtain lower bounds on $\botn (\bullet )$, then its monotonicity (\Cref{thm:monotonicity}) suggests that Limitation (1) need not be a concern in practice.
Limitation (2) can also be addressed with rigorous, certified numerical path-tracking---see eg.~\cite{guillemot,lee-homotopy} and references therein.

Our work is close in spirit to recent studies of decomposable branched covers, and various proposals for ``optimal" numerical homotopy continuation methods.
\Cref{subsec:optimal} reviews this related work in the context of Galois width.
We must also mention the work~\cite{hartley}, discussed further in~\Cref{subsec:avis}, which inspired our computational model.
To the best of our knowledge, Galois width has not been explicitly studied in the literature.

The paper concludes by discussing three topics where Galois width has, at least in hindsight, already played an important role: metric algebraic geometry (\Cref{subsec:mag}), algebraic statistics (\Cref{subsec:astat}), and algebraic vision (\Cref{subsec:avis}.)
We hope the perspective and tools developed here will aid readers pursuing their own applications.

\section{Galois Width}\label{sec:bottleneck}

\subsection{Preliminaries}\label{subsec:prelims}

We first establish notation and state (mostly standard) results that will play a key role in what follows.
We denote by $\QQalg \subset \CC $ the usual field of algebraic numbers.
An intermediate field $\QQ \subset \KK \subset \QQalg $ is called a number field if $\KK / \QQ $ is a finite extension.
Many of our results will remain intact after replacing $\QQalg / \QQ $ with more a general field extension.
However, this level of generality is not necessary for the applications given later.

In our setting, the compositum $\KK \LL$ of two number fields $\KK$ and $\LL$ is the intersection of all subfields of $\QQalg $ containing both $\KK $ and $\LL$.
Any element of $\KK \LL $ may be written as an $\LL $-linear combination of elements in $\KK .$

\begin{prop}\label{prop:field-index-composite-bound}
Let $\FF, \KK , \LL $ be number fields with $\FF \subset \KK.$ Then 
\[
[ \KK : \FF ] \ge [\KK \LL : \FF \LL ].
\]
\end{prop}
\begin{proof}
Let $\alpha $ be a primitive element for $\KK / \FF $, so that the powers $\alpha^i,$ for $i=0,\ldots , [\KK : \FF ]-1,$ form a $\FF $-basis for $\KK .$
Promoting these powers under the inclusion map $\KK \hookrightarrow \KK \LL$ produces a spanning set for $\KK \LL $ over $\FF \LL .$
\end{proof}

We now state a useful group-theoretic fact.
The argument relates to the notion of a stabilizer chain (see eg.~\cite[\S 1.13]{cameron}.) 

\begin{lemma}\label{lem:stabchain-perm}
Let $G\le S_d$ be a transitive permutation group.
Then there exists a chain of subgroups 
\[
G = H_0 \ge  H_1 \ge H_2 \ge \ldots \ge H_d = \operatorname{id} 
\]
such that $d = [H_0 : H_1] > [H_i:H_{i+1}] $ for all $i=1,\ldots , d-1.$
\end{lemma}

\begin{proof}
Set $W_0 = S_d$, and for $i=1 \ldots , d,$ define subgroups
\begin{align}
W_i &= \displaystyle\bigcap_{j=1}^i \operatorname{Stab}_{S_d} (i) \le S_d,\\
H_i &= W_i \cap G \le G. 
\end{align}
We claim that the chain 
\begin{equation}\label{eq:H-chain-simple}
G = H_0 \ge H_1 \ge \ldots \ge 
H_d = \operatorname{id}
\end{equation}
satisfies the desired conclusion.
First note $[H_0 : H_1]=d$ by the orbit-stabilizer theorem. 
Then, comparing the chain in $G$ against the corresponding chain in $S_d,$ namely
\begin{equation}\label{ew:W-chain-simple}
S_d = W_0 \ge W_1 \ge \ldots 
\ge W_d = \operatorname{id} ,
\end{equation}
we bound the remaining indices for $i=1, \ldots , d$:
\begin{align*}
[H_i : H_{i+1}] = \displaystyle\frac{|W_i \cap G|}{|W_{i+1} \cap G|} 
&= \displaystyle\frac{|W_{i+1} G|}{|W_i G|} \cdot 
\displaystyle\frac{|W_i||G|}{|W_{i+1}| |G|}\\
&= \displaystyle\frac{|W_{i+1} G|}{|W_i G|} \cdot [W_i : W_{i+1} ]\\
&\le [W_i : W_{i+1} ] = \displaystyle\frac{(d-i)!}{(d-i-1)!} = d-i < d.
\end{align*}
\end{proof}

\begin{cor}\label{cor:stabchain-simple}
Let $G$ be a finite simple group with a subgroup $H\subset G.$ 
Then there exists a chain of subgroups 
\[
G = H_0 \ge H = H_1 \ge H_2 \ge \ldots \ge H_d = \operatorname{id}
\]
such that $[H_0 : H_1]> [H_i:H_{i+1}]$ for all $i=1,\ldots , d-1.$
\end{cor}

\begin{remark}\label{rem:corefree}
More generally, although we will not need it, the conclusion of~\Cref{cor:stabchain-simple} holds when $H$ is \emph{core-free} in $G$, meaning $\cap_{g\in g} g^{-1} H g = \operatorname{id}.$
\end{remark}

\begin{proof}
Set $d= [G:H_1],$ and note that $G$ acts transitively on left cosets of $H$, giving a group homomorphism $G \to S_d.$
Since $G$ is simple, we may identify it with its image under this homomorphism, and $H_1$ with the stabilizer of a point.
The conclusion now follows as in~\Cref{lem:stabchain-perm}.
\end{proof}

The interplay between abstract finite groups and their permutation representations is a reoccurring theme in this work.
In particular, we will need to collect some useful results on the well-studied topic of minimum faithful permutation degrees.

\begin{definition}\label{def:minimal-faithful-degree}
For any finite group $G,$ we define $\mu (G)$ to be the minimum degree of a faithful permutation representation of $G.$
That is, \[
\mu (G) = \min \left\{ d \in \ZZ_{\ge 0} \mid \exists \, \varphi \in \operatorname{Hom} (G, S_d) \text{ w/ } \ker \varphi = \textrm{id} \right\} .
\]
Such a representation is \emph{minimal} if its degree is $\mu (G).$
\end{definition}
Faithful permutation representations remain faithful when restricted to subgroups: thus, if $H\le G$ is any subgroup, we have
\begin{equation}\label{eq:mu-subgroup}
\mu (H) \le \mu (G).    
\end{equation}
One should bear in mind that a minimal faithful permutation representation $\varphi : G \to S_d$ need not give a transitive subgroup of $S_d.$
For example, if $G= \ZZ / 2 \ZZ \times \ZZ / 3 \ZZ,$ then $\mu (G) =5,$ and if $\varphi :G \to S_5$ is minimal then $\varphi (G)$ is generated by disjoint cycles of length $2$ and $3.$

Fortunately, when $G$ is simple, minimality implies transitivity.

\begin{prop}\label{prop:simple-perm-degree}
If $G$ is a finite simple group with $\mu (G) = d,$ 
then 
\[
\mu (G) = \displaystyle\min_{M < G \text{ maximal}} [G : M] = \displaystyle\min_{H < G } \, [G : H],
\]
and any minimal faithful representation of $G$ is transitive.
\end{prop}

A short proof of~\Cref{prop:simple-perm-degree} is given in~\cite[Proposition 1 (i)]{johnson}.

In contrast to the case of subgroups, the behavior of $\mu (\bullet )$ for quotients is more subtle. This behavior was studied by Kov\'{a}cs and Praeger, who proved the following result.

\begin{theorem}\label{thm:perm-degree-quotient}
(\cite[Theorem 1]{praeger}) 
Let $H$ be a finite group with a normal subgroup $N\unlhd H.$
If $H/N$ has no nontrivial abelian normal subgroup (in particular, if $H/N$ is simple), then $\mu (H/N) \le \mu (H).$
\end{theorem}

\Cref{thm:perm-degree-quotient} is key to our proof of~\Cref{thm:monotonicity}.
Finally, we will use the following criterion recognizing when a transitive group is full-symmetric.

\begin{lemma}\label{lem:stabilizer-better}
Let $G \subset S_d$ act transitively on $\{1 , \ldots , d\}.$
If $S_k \times S_{d-k} \subset G$, where $1\le k<d/2,$ then $G=S_d$.
\end{lemma}
\begin{proof}
Maximal subgroups of $S_d$ which are intransitive are precisely those of the form $S_k \times S_{d-k}$, where $1\le k< d/2$.
See eg.~\cite[Theorem 4.8]{cameron}.
\end{proof}
\begin{remark}\label{rem:symmetric-failure}
The conclusion of~\Cref{lem:stabilizer-better} fails when $n$ is even, $k=d/2,$ and $G = S_{k} \wr S_2$ is the \emph{wreath product} containing all permutations in $S_{d}$ that preserve the partition $\{1 ,\ldots , k \} \sqcup \{ k+1 , \ldots , d \}$.
\end{remark}

\subsection{Algebraic Computation Model}\label{subsec:model}

We now address the question of how to define the complexity of an algebraic number $\alpha \in \QQalg $.
To answer this question, we must define a model of computation and a class of admissible algorithms.
We begin with a slightly informal description.
Our algorithms will execute a finite number of steps $m,$ and maintain a \emph{working field} $\FF_k \subset \QQalg$ for each step $k=1, \ldots , m$.
We initialize $\FF_0\gets  \mathbb{Q}.$
At each step $k\ge 1,$ an algorithm can perform one of the following two ``oracle steps":
\begin{itemize}
\item[1.] Perform arithmetic in the field $\FF_{k-1}$ (and set $\FF_k \gets \FF_{k-1}$), or
\item[2.] Compute a root $\alpha_k \in \QQalg$ of some polynomial $p_k(x)\in \FF_{k-1} [x],$ 
and extend the working field $\FF_k \gets \FF_{k-1} (\alpha_k).
$
\end{itemize}
Finally, the output of the algorithm is given by $\alpha_m.$
We note that any arithmetic operation of type 1 is in fact a special case of the root-finding operation of type 2; indeed, for any operation $\textrm{op} \in \{ + , -, \times , \div \} ,$ the polynomial $x - (\alpha \textrm{ op } \beta) \in \FF_{k-1} [x]$ has a unique root given by $\alpha \textrm{ op } \beta .$

The next definition, although more abstract, captures all of the essential data needed to analyze the complexity of an algorithm in this model.
It is exactly the same as in~\cite[Definition 2.2]{hartley}.
\begin{definition}\label{def:algorithm}
Let $\alpha \in \QQalg$. An \emph{algorithm} is tower of field extensions
\begin{equation}\label{eq:tower}
\QQ = \FF_0 \subset \FF_1 \subset \FF_2 \subset \cdots \subset \FF_m \ni \alpha .
\end{equation}
\end{definition}

We will measure the cost of the $k$-th step in algorithm by the degree of the working field extension $[\FF_k : \FF_{k-1}]$.
We point out out that, in order to obtain an interesting theory, it is crucial to allow the working field to depend on the results of previous computations; if we were to work in the simpler setting where the working field is always the rational numbers, then the cost of an algebraic number would just be its degree over $\QQ .$

We now propose measuring the total cost of an algorithm in this model by its most costly step(s).
We would like, then, for any $\alpha \in \QQalg ,$ to understand which algorithms computing $\alpha $ minimize the most costly step.
This leads to the following definition.

\begin{definition}\label{def:bottleneck-number}
The \emph{Galois width} of $\gr{\alpha } \in \QQalg$ is the quantity
\[
\botn (\gr{\alpha} ) = 
\gl{\displaystyle\min_{\substack{\gr{\text{algorithms}}\\\gr{\QQ}  = \gr{\FF_0} \subset 
\cdots 
\subset 
\gr{\FF_m} \ni 
\gr{\alpha}}}}
\left( 
\gl{\max_{0\le i < m}} \, \, [\gr{\FF_{i+1}} : \gr{\FF_i}]
\right).
\]
\end{definition}
We now define a parallel notion for a finite group $G$.
We define the \emph{cost} of a chain of subgroups 
\[
G \ge H_0 \ge \ldots \ge H_n   
\]
to be the maximum value of the indices $[H_{i} : H_{i+1}]$ for $i=0,\ldots , n-1.$
To compute the cost of such a subgroup chain, we may assume without loss of generality that all inclusions are strict, $H_i > H_{i+1}.$
The \emph{length} of the chain~\eqref{eq:chain} is defined to be $n$, and we say this chain is \emph{maximal} if it cannot be refined to a longer chain by inserting strict subgroups.
\begin{definition}\label{def:bottleneck-group}
The \emph{Galois width} of a finite group $G$ is the minimum cost taken over all maximal chains of subgroups
\begin{equation}\label{eq:chain}
G = H_0 > \ldots > H_n = \operatorname{id}   .
\end{equation}
We denote this quantity by $\botn(G).$    
\end{definition}

We now relate~\Cref{def:bottleneck-number,def:bottleneck-group}.

\begin{theorem}\label{thm:complexity}
If $f\in \QQ [x]$ is an irreducible polynomial with Galois group $G,$ and $\alpha $ is any root of $f$, then
\[
\botn (\alpha) = \botn (G).
\]
\end{theorem}
\begin{remark}
Although the Galois correspondence makes~\Cref{thm:complexity} plausible, some care is required.
Note that the minimum in~\Cref{def:bottleneck-group} is taken over a finite set, whereas the set of algorithms in~\Cref{def:bottleneck-number} is infinite. Moreover, the final field $\FF_m$ in an algorithm computing $\alpha $ need not be inclusion-wise comparable with the splitting field of $f.$
\end{remark}
\begin{proof}
Throughout, we let $\FF \subset \QQalg$ denote the Galois closure of the extension $\QQ (\alpha) / \QQ .$
First, we observe that $\botn(\alpha )\le \botn (G)$. 
Indeed, under the Galois correspondence, a maximal chain of subgroups of $G$ attaining the minimum cost $\botn (G)$ corresponds to a chain of fields
\begin{equation}\label{eq:tower-thm1-1}
\QQ = \FF_0 \subset \FF_1 \subset \cdots \subset \FF_m = \FF 
\end{equation}
such that $[\FF_i : \FF_{i-1}] \le \botn (G) $ for all $i.$ 

It remains to prove $\botn (\alpha ) \ge \botn (G).$ 
Consider any tower of fields 
\begin{equation}\label{eq:tower-F'}
\QQ = \FF_0' \subset \FF_1 ' \subset \cdots \subset \FF_m ' \ni \alpha 
\end{equation}
attaining the minimum cost $\botn (\alpha ).$
Let $\FF '$ denote the Galois closure the extension $\FF_m ' / \QQ $, with $G'$ its Galois group.
Since $\FF \subset \FF '
$, the Galois correspondence implies that $G\cong G'/N'$ for some normal subgroup of $N' \unlhd G'.$
Corresponding to~\eqref{eq:tower-F'} are both the subgroup chain in $G'$ given by
\begin{equation}\label{eq:chain-G'}
G' = H_0 ' \ge \cdots \ge H_m '
\end{equation}
as well as the subgroup chain in $G$ given by 
\begin{equation}\label{eq:chain-G}
G = H_0 \ge \cdots \ge H_m  ,
\end{equation}
where $H_i := N ' H_i ' / N'.$ 
The latter chain corresponds to the tower of fields 
\begin{equation}\label{eq:tower-F'-intersect}
\QQ = \widetilde{\FF}_0  \subset \widetilde{\FF}_1 \subset \cdots \subset \widetilde{\FF}_m  ,
\end{equation}
where $\widetilde{\FF}_i = \FF_i ' \cap \FF .$
Since $H_i \cong H_i' / (H_i' \cap N)$, we have 
\[
[H_i : H_{i+1} ] = \displaystyle\frac{[H_i ' : H_{i+1}']}{[H_i' \cap N' : H_{i+1}' \cap N']} \le [H_i ': H_{i+1}'] \le \botn (\alpha),
\]
and it will suffice to prove that the subgroup chain~\eqref{eq:chain-G} has cost at least $\botn (G).$
To do so, we extend the tower~\eqref{eq:tower-F'-intersect} as follows: let $\beta $ be a primitive element for the extension $\widetilde{\FF}_m / \QQ $, and $\gamma = g \cdot \beta$ one of its Galois conjugates.
We may then form the tower of composite extensions
\begin{equation}\label{eq:tower-F'-conjugate}
\QQ (\beta) = \widetilde{\FF}_m \subset \widetilde{\FF}_m \left( g\cdot \widetilde{\FF}_1 \right) \subset \cdots \subset \widetilde{\FF}_m \left( g \cdot \widetilde{\FF}_m \right) = \QQ (\beta, \gamma ).
\end{equation}
By~\Cref{prop:field-index-composite-bound}, we have
\[
\left[ \widetilde{\FF}_m \left( g\cdot \widetilde{\FF}_{i+1} \right) : \widetilde{\FF}_m \left( g\cdot \widetilde{\FF}_i \right)\right]
\le 
\left[g\cdot \widetilde{\FF}_{i+1} : g\cdot \widetilde{\FF}_i\right]
=
\left[\widetilde{\FF}_{i+1} : \widetilde{\FF}_i\right].
\]
If $\QQ (\beta , \gamma ) \subsetneq \FF ,$ we construct yet another chain as in~\eqref{eq:tower-F'-conjugate}, but starting from $\QQ (\beta , \gamma )$.
Continuing this process, we eventually obtain a tower of fields from $\QQ $ to $\FF,$ corresponding to an extension of the subgroup chain~\eqref{eq:chain-G},
\begin{equation}\label{eq:final-chain}
G = H_0 \ge \cdots \ge H_m \ge \cdots \ge H_l = \operatorname{id},
\end{equation}
with the property that $[H_i : H_{i+1}] \le \botn (\alpha)$ for all $i\ge m$.
Thus, the cost of both subgroup chains~\eqref{eq:chain-G} and~\eqref{eq:final-chain} is at most $\botn (\alpha ).$
Since the cost of~\eqref{eq:final-chain} is at least $\botn (G),$ this completes the proof.
\end{proof}

\subsection{Properties of Galois Width}\label{subsec:complexity}

Having now established the Galois width as the principal quantity of interest in our computation model, we turn to the problem of computing this quantity.
Since~\Cref{def:bottleneck-group} is essentially a minimax formulation, it would be natural to consider methods from combinatorial optimization---and indeed, a subgroup chain in $G$ attaining the minimum cost $\botn (G)$ may in principal be found using network flow on the lattice of subgroups.

However, the graph of all subgroups (or even conjugacy classes thereof) is much too cumbersome to work with in practice.
Thus, it is desirable to express $\botn (G)$ in terms of group-theoretic primitives.
\Cref{thm:compute-bot} expresses $\botn (G)$ in terms as the maximum over all minimal faithful permutation degrees of the composition factors of $G.$
We also use this section to prove ~\Cref{thm:monotonicity}, a monotonicity property for the Galois width.

We begin with a key splitting property.

\begin{prop}\label{prop:bottleneck-normal-recursion}
For any normal subgroup $N\unlhd G$, we have that
\begin{equation}\label{eq:bottleneck-normal-subgroup-formula}
\botn (G) = \max \left( \botn (G/N), \botn (N) \right).
\end{equation}
\end{prop}
\begin{proof}
First, we argue that the left-hand-side in~\eqref{eq:bottleneck-normal-subgroup-formula} is bounded by the right-hand-side.
To see this, note that the lattice isomorphism theorem identifies any maximal chain of subgroups in $G/N$ with a chain of subgroups $G > \ldots > N.$ 
Such a chain attaining the minimum cost $\botn (G/N)$ may be prepended to a maximal chain of subgroups of $N$ which attains the minimum cost $\botn (N)$. 
This produces a maximal chain in $G$ whose cost equals $\max \left( \botn (G/N), \botn (N) \right)$.
Thus $\botn (G) \le \max \left( \botn (G/N), \botn (N) \right).$
\\\\
We now prove the reverse inequality. 
Consider a maximal chain~\eqref{eq:chain} whose cost attains the minimum $\botn (G).$
We may ``reroute" this chain through $N$ using intersections and subgroup products as follows:
\[
G = N H_0 \ge \ldots  \ge N H_n = N = H_0 \cap N \ge \ldots  \ge H_n \cap N = \text{id}.
\]
The cost of this rerouted chain is at least $\max \left(\botn (G/N), \botn (N) \right).$
From the product formula
\begin{align*}
[H_{i} : H_{i+1}]=
[H_{i} \cap N  : H_{i+1} \cap N ] [N H_{i} : N H_{i+1}] ,
\end{align*}
and the bound $xy \ge \max (x,y)$, valid for any positive integers $x$ and $y,$
it follows that the rerouted chain also attains the minimum cost $\botn (G)$.
Thus $\max \left(\botn (G/N), \botn (N) \right) \le \botn (G).$

\end{proof}

\Cref{prop:bottleneck-normal-recursion} shows that the Galois width is determined by its values on simple groups.
Moreover, when the group $G$ is simple, it turns out that $\botn (G)$ equals the minimum (faithful) permutation degree $\mu (G).$
This in turn implies a clean formula for the Galois width of any finite group $G$.
Recall that the \emph{Jordan-H\"{o}lder Theorem} states that $G$ always has a \emph{composition series}: that is, a chain
\begin{equation}\label{eq:composition-series}
G = N_0 \unrhd N_1 \cdots \unrhd N_n = \operatorname{id}
\end{equation}
such that all successive quotient groups $N_{i} / N_{i+1}$, called \emph{composition factors}, are simple.
The chain~\eqref{eq:composition-series} is \emph{subnormal}: ie., $N_j$ need not be normal in $N_i$ for $i < j-1$.
Multiple composition series for a group may exist; however, the multiset of composition factors is uniquely determined.

\begin{theorem}\label{thm:compute-bot}
If $G$ is any finite group with composition series~\eqref{eq:composition-series}, then
\begin{equation}\label{eq:bot-compositive}
\botn (G) = \displaystyle\max_{0\le i \le n-1} \mu (N_{i} / N_{i+1}).    
\end{equation}
\end{theorem}
\begin{proof}
It is enough to prove the result when $G$ is simple---the general case then follows by repeated application of~\Cref{prop:bottleneck-normal-recursion}.
Let $M < G$ be a maximal subgroup.
Applying~\Cref{cor:stabchain-simple}, we have that
\begin{equation}\label{eq:maximal-bound}
\botn (G) \le [G : M].    
\end{equation}
Moreover, the reverse inequality must hold for any maximal subgroup $M$ which appears in a chain attaining the minimum cost $\botn (G).$
Such an $M$ is clearly a subgroup of minimum index in $G$.
Thus, using~\Cref{prop:simple-perm-degree}, we conclude that $$\botn(G) = [G:M]=\mu (G).$$
\end{proof}
\Cref{thm:compute-bot} allows us to easily evaluate $\botn (G)$ for many familiar families of groups.
For example, if $G$ is solvable, then $\botn (G)$ equals the largest prime dividing the group order $|G|.$
For the symmetric and alternating groups, we have the composition series $S_n \unrhd A_n \unrhd \text{id}$.
Recall that, when $n\ge 5,$ the group $A_n$ is simple and has no subgroup of index less than $n.$ 
For $A_4,$ on the other hand, we have the subgroup chain
\[
A_4 \ge \mathbb{Z}/2\mathbb{Z} \times \mathbb{Z}/2\mathbb{Z} \ge \mathbb{Z}/2\mathbb{Z} \ge \text{id},
\]
which attains the minimum cost of 3. Altogether, this gives 
\begin{equation}\label{eq:bot-sn-an}
\botn (\gr{S_n} ) = \botn (\gr{A_n} ) = \begin{cases}
          \gr{3} \quad \text{if } \gr{n} = \gr{4}, \\
          \gr{n} \quad \text{else.}
      \end{cases}
\end{equation}
For a Galois-theoretic interpretation of formula~\eqref{eq:bot-sn-an}, recall that the roots of a general univariate quartic equation may be expressed in terms of the roots of its \emph{resolvent cubic} and additional square roots thereof. 
For larger values of $n,$ no such reduction is generally possible.

We now prove that $\botn (\bullet )$ is monotonic with respect to inclusion.
\begin{theorem}\label{thm:monotonicity}
\begin{equation}\label{eq:subgroup-bound}
\botn (H) \le \botn (G) \quad \text{ for any subgroup } H \le G.
\end{equation}    
\end{theorem}
\begin{proof}
We induct twice: first on the order of $G$, then on the order $H.$
The base cases where $G$ or $H$ are trivial follow trivially.

Consider first the case where $H$ contains a nontrivial normal subgroup of $G$, say $N\unlhd G.$ Then $H / N \le G / N,$ and we have
\begin{align*}
\botn (H) &= \max \left( \botn (H/N), \botn (N) \right) \tag{by~\Cref{prop:bottleneck-normal-recursion}}\\
&\le \max \left(\botn (G/N), \botn (N) \right) \tag{by induction}\\
&= \botn (G). 
\end{align*}
More generally, suppose $N' \unlhd G$ is a nontrivial normal subgroup of $G,$ and set $N=N'\cap H.$ 
Either $ N\ne \operatorname{id}$, in which case the previous argument works, or $H\cong N'H/N'.$
The latter case yields
\[
\botn (H) = \botn (N'H/ N') \le \botn (N'H) \le \botn (G),
\]
where the last inequality follows from the case previously considered.

It remains to treat the case where $G$ is simple.
Taking now $N \unlhd H$ to be a proper normal subgroup with $H/N$ simple, we then have
\begin{align}
\botn (H/N) &= \mu (H/N) \tag{\Cref{thm:compute-bot}} \nonumber \\
&\le \mu (H) \tag{\Cref{thm:perm-degree-quotient}} \nonumber \\
&\le \mu (G) \tag{eq.~\eqref{eq:mu-subgroup}} \nonumber \\
&= \botn (G) \tag{\Cref{thm:compute-bot}} \nonumber.
\end{align}
By induction, $\botn (N) \le \botn (G).$ Thus, combining~\Cref{prop:bottleneck-normal-recursion} and the inequality above, we conclude that
\[
\botn (H) = \max (\botn (N), \botn(H/N)) \le \botn (G).
\]
\end{proof}

\section{Solving Algebraic Equations}\label{sec:applications}

\subsection{Branched Covers and Thin Sets}\label{subsec:br-thin}

In this section, we pass from univariate polynomials to multivariate systems of equations.
For a more detailed exposition of Galois groups arising in polynomial system solving, we recommend the recent survey~\cite{sottile2021galois}.

Consider an algebraic variety $X\subset \CC^m \times \CC^n$ whose points $(p,x)$ are \emph{problem-solution pairs.}
Solutions to a particular problem instance $p\in \CC^m,$ then, may be identified with points in the fiber $\pi^{-1} (p),$ where
\begin{equation}\label{eq:branched-cover}
\pi : X \to \CC^m, \quad 
(p, x) \mapsto p
\end{equation}
projects $X$ onto the space of problems.
We assume:
\begin{enumerate}
    \item[\textbf{A1}] For a generic problem instance $\gl{p}\in \gl{\CC^m},$ there exists a finite, nonzero \gr{number of solutions}, $\gr{d}=\# \pi^{-1} (p)$. More precisely, for all $p$ in some nonempty Zariski-open $U \subset \CC^m$ (depending on $\pi$), the set $\pi^{-1} (p)$ consists of $d\ge 1$ distinct points.
    \item[\textbf{A2}] The variety $\gr{X}$ is \emph{irreducible}.
\end{enumerate}
In other words, the map $\pi$ is a branched cover. More generally, we define a branched cover $\pi : X \dashrightarrow Z$ to be a dominant rational map between quasiprojective varieties of equal dimension over $\CC$.
The \emph{degree} of $\pi$, denoted $\deg (\pi ),$ is the number of points $\# \pi^{-1} (z)$ in a generic fiber.

Assumption \textbf{A1} essentially requires that problems are well-posed.
Assumption \textbf{A2}, although strong, is often satisfied in practice.
A common theme among the examples to be presented is that \gl{problems} may be written as rational functions of their \gr{solutions}, in which case $X$ would be the graph of a rational map $\CC^m \dashrightarrow \CC^m$.
Much of our ensuing discussion could be adapted to reducible varieties by working component-by-component.

Since $X$ is irreducible, we may consider a primitive element $f$ for the extension of function fields $\gr{\CC (X)} / \gl{\CC (p_1, \ldots , p_m)},$
\begin{equation}\label{eq:eliminant}
f(y; \gl{p} ) = \gr{y^d} + c_1 (\gl{p}) \gr{y^{d-1}} + \ldots + c_d (\gl{p}).
\end{equation}
The \emph{Galois/monodromy group} $\operatorname{Gal} (\pi)$ may be defined as the Galois group of the polynomial $f$ over $\CC (p_1, \ldots , p_m).$
Alternatively, one may define $\operatorname{Gal} (\pi)$ via the monodromy action of $\Gal (\pi)$ on a fiber $\pi^{-1}(p)$ of size $\deg (\pi ).$
This is a group homomorphism $\pi_1 (U ; p) \to S_d$, where $\pi_1 (\bullet )$ denotes the topological fundamental group.
The permutations in $\Gal (\pi)$ are obtained by lifting paths in $U$ based at $p$ up to $X.$
Numerical homotopy continuation methods that compute approximations of these lifts are the basis of popular heuristics for computing $\Gal (\pi )$---see eg.~\cite{DUFF-MONODROMY,del-campo-rodriguez,leykin-schubert,breiding-timme,jir-galois}.
Assumption \textbf{A2} implies that $\Gal (\pi )$ acts \emph{transitively} on $\pi^{-1} (p)$.

A third assumption often holding in practice is that the variety $X$ is defined by polynomials with rational coefficients:
\begin{itemize}
    \item[\textbf{A3}] The variety $X$ is defined over $\mathbb{Q}.$
\end{itemize}
In this case, a suitable $f$ may be computed (at least in principle) using Gr\"{o}bner bases~\cite{gianni}.
Let $I \subset R = \QQ (p) [x_1,\ldots , x_n]$ be the ideal obtained from the vanishing ideal of $X$ by extension of scalars.
This is a polynomial ideal of dimension $0$ and degree $d.$
For generic $(a_1,\ldots, a_n) \in \QQ^n,$ consider the linear form $y=a_1 x_1 + \ldots + a_n x_n.$ 
We endow the polynomial ring $R=\QQ (p) [x_1, \ldots , x_{n-1}, y]$ with the lexicographic order $x_1> \cdots > x_{n-1} > y.$
The elimination ideal $I \cap \QQ (p) [y] $ is then generated by $f$ of the form~\eqref{eq:eliminant}.

From our preceding discussion, it becomes natural to consider three types of Galois group: $\Gal (\pi),$ the Galois group of the univariate polynomial $f(y;p)\in \QQ (p) [y]$, and Galois groups of specializations $f(y;\tilde{p})$ for particular problem instances $\tilde{p}\in \QQ^m.$ 
When needed, we denote the latter two groups by $\Gal_{\QQ (p)} (f)$ and $\Gal_{\QQ} (f(y;\tilde{p})),$ respectively.
The theory developed in~\Cref{sec:bottleneck} applies most directly to the last of these types.
In general, the three types of groups may all be different.

\begin{example}\label{ex:a3}
Consider the branched cover $\pi : X \to \CC^1$ given by
\[
X = \{ (p, x) \in \CC^1 \times \CC^1 \mid f(x; p)=0 \} , \quad 
f(x;p) = x^3-p.
\]
We have $\Gal (\pi)\cong A_3$, yet the Galois group of $f\in \QQ (p)[x]$ is $S_3.$
Although $S_3$ is also typically the Galois group after specialization, this is not the case for certain unlucky specializations like $\tilde{p}=0$, in which case the Galois group is trivial.
This example shows that $\Gal (\pi )$ and the specialized Galois groups are generally not comparable inclusion-wise.
\end{example}

The notion of a thin set defined in~\cite[Ch.~3]{serre}
captures the essence of ``unlucky" specializations such as in the example above.
We say a $\Delta \subset \QQ^m$ has \emph{type C1} if it is contained in a proper Zariski-closed subset of $\CC^m$, and \emph{type C2} if there is a branched cover $\pi :X \to \CC^m$ with $\deg (\pi ) \ge 2$ and $\Delta \subset \pi (X_{\QQ})$ (here $X_{\QQ}$ denotes the set of $\QQ$-valued points in $X.$)

\begin{definition}\label{def:thin-set}
A set $\Delta \subset \QQ^m$ is said to be \emph{thin} if it is contained in a finite union of sets of types C1 or C2.
\end{definition}

\begin{example}\label{ex:thin}
Continuing with~\Cref{ex:a3}, we show that the set
\[
\Delta = \{ p \in \QQ \mid \Gal_\QQ ( x^3 - p  ) \ne S_3 \} .
\]
is thin.
If $\tilde{p}$ is chosen such that $x^3-\tilde{p}$ has a rational root, then the Galois group $\Gal_\QQ (x^3-\tilde{p})$ is contained in the subgroup $S_2$: we define
\[
X^{S_2} = \{ (p,a,b,c) \in \CC^4 \mid x^3 - p  = (x-a) (x^2+bx + c) \}.
\]
The condition that $p\in X_\QQ^{S_2}$ may be found by equating coefficients in $x$,
\[
p = ac = a^2 b = a^3.
\]
Furthermore, the map $\pi : X^{S_2} \to \CC^1$ defined by $(p,a,b,c) \mapsto p$ is a branched cover of degree $3.$
Thus, we see that $\pi (X^{S_2}_\QQ)\subset \QQ$, the set of rational perfect cubes, is thin.
One may also rule out $\Gal_\QQ (x^3-\tilde{p}) \subset A_3$ by computing the discriminant, $27p^2=0 \Rightarrow p=0.$
Since $A_3$ and $S_2$ are the maximal subgroups of $S_3,$ we conclude that $\Delta = \pi (X^{S_2}_\QQ)$ is thin.
\end{example}

\Cref{ex:thin} generalizes significantly.

\begin{theorem}\cite[Proposition 3.3.5]{serre}
For $f$ as in equation~\eqref{eq:eliminant}, there exists a thin set $\Delta \subset \QQ^m$ such that if $\tilde{p}\notin \Delta ,$ then $\tilde{p}$ is not a pole of any of the rational functions $c_1(p), \ldots , c_d(p),$ the polynomial $f(y;\tilde{p})\in \QQ [y]$ is irreducible over $\QQ ,$ and its Galois group equals that of $f$ over $\QQ (p).$\label{thm:hit}
\end{theorem}

We define the \emph{Galois width} $\botn (\pi)$ of a branched cover $\pi $ in the obvious manner to be the Galois width of its Galois/monodromy group $\Gal (\pi )$ as in~\Cref{def:bottleneck-group}.
Using~\Cref{thm:hit}, we may relate the three different incarnations of Galois width studied thus far.

\begin{cor}\label{cor:hit}
With $\pi : X \to \CC^m$ and $f$ satisfying assumptions \textbf{A1}--\textbf{A3} above, there exists a thin set $\Delta \subset \QQ^m$ with the following property: for any $\tilde{p} \notin \Delta $ and any $\alpha $ with $f(\alpha ; \tilde{p}) =0,$ we have 
\begin{equation}\label{eq:thm-lower-bound}
\botn \left( \pi   \right) \le \botn \left( \Gal_{\QQ (p)} (f)\right) = \botn (\alpha ) .
\end{equation}
Furthermore, for any $\tilde{p} \in \QQ^m $ and any $\alpha $ with $f(\alpha ; \tilde{p}) =0,$ we have 
\begin{equation}\label{eq:thm-upper-bound}
\botn (\alpha )  \le \botn(\Gal_{\QQ (p)} (f)).
\end{equation}
\end{cor}

\begin{proof}
Taking $\Delta $ to be the thin set guaranteed by~\Cref{thm:hit}, we have for any $\tilde{p}\notin \Delta $ that $\botn (\alpha) = \botn \left( \Gal_{\QQ (p)} (f)\right).$
Now, $\Gal (\pi ) \le \Gal_{\QQ (p)} (f),$ as explained eg.~in~\cite[\S 1.2]{sottile2021galois}; in fact, $\Gal (\pi ) \unlhd \Gal_{\QQ (p)} (f),$ since the base of $\pi$ is the rational variety $\CC^m .$
Thus~\Cref{thm:monotonicity} gives us the needed inequality in~\eqref{eq:thm-lower-bound}.
Inequality~\eqref{eq:thm-upper-bound} may be deduced from the proof of~\cite[Proposition 3.1.1]{serre}, which expresses $\Delta $ as a union of thin sets for which some proper subgroup $H<\Gal_{\QQ (p)} (f)$ occurs as the specialized Galois group.
\end{proof}

It is worth emphasizing that the thin set appearing in~\Cref{cor:hit} is ``small" in multiple senses of the word.
Indeed, several number-theoretic results show that thin sets are of small \emph{natural density}.
For instance,~\cite[Theorem 3.4.4]{serre} provides an asymptotic count of integral points in $\Delta $,
\begin{equation}\label{eq:density}
\# \left( \Delta \cap \ZZ^m \cap [-N, N] \right) = O \left( N^{m-\frac{1}{2}} \log N \right) \ll O(N^m).
\end{equation}

\Cref{cor:hit} also offers takeaways for practitioners of both symbolic and numerical computation.
In the symbolic approach, a family of problems may be studied by fixing a particular problem instance $\tilde{p}\in \QQ^m$ and computing $\Gal_{\QQ} (f(x;\tilde{p}))$.
When the latter can be carried out algorithmically (eg.~using Stauduhar's method~\cite{stauduhar}), either of the inequalities in~\Cref{thm:hit} may then provide information; the Galois width for a specific problem instance provides a lower bound for the generic case $\tilde{p}\notin \Delta .$
In other cases, heuristics such Frobenius elements may be used to show $\Gal_{\QQ (p)} (f)=S_d$ (see eg.~\Cref{ex:ml-linear}.)

Numerical approaches typically compute a subgroup of monodromy permutations $H\le \Gal (\pi )$.
One may also hope for equality; the branch point method of~\cite{jir-galois} guarantees this, although it may be slower than other heuristics and more difficult to certify.
Once more,~\Cref{thm:monotonicity} shows that even a partial computation of $\Gal (\pi) $ can provide a useful lower bound on the complexity of solving a generic problem instance.
We note that inequality~\eqref{eq:thm-upper-bound} may also be applied when the equality $\botn( \Gal_{\QQ (p)} (f)) = \botn (\Gal (\pi ))$ is known \emph{a priori}. 
This often arises when $\deg (f) \ge 5$ and $\Gal_{\QQ (p)} (f)$ is full-symmetric; this forces $\Gal (\pi)$ to either be full-symmetric or full-alternating, and in either case the Galois widths are equal.

Numerical methods have (at least) one decisive advantage over symbolic computation: complete equations describing $X$ are not needed, much less the primitive element~\eqref{eq:eliminant}. All that we need in practice are $n$ equations cutting out $X$ locally (these may be polynomial or rational) and some method for sampling generic points on $X.$

\subsection{Homotopies and Optimal Solving}\label{subsec:optimal}

In the homotopy continuation literature, there have been various proposals of ``optimal homotopies".
Since rigorous analysis of numerical path-tracking is very difficult, the usual objective is to minimize number of solution paths tracked.
Even for such a coarse measure of complexity, none of the various proposals is completely satisfactory.
Our present study was partially motivated by the author's desire to settle the question, once-and-for-all, of what makes a homotopy optimal.
This section is partly a reflection on his failure to find clear answers.

Polyhedral homotopies constitute a well-studied class of homotopies sometimes considered to be optimal. 
They arise when solving $n$ equations in $n$ unknowns, where each equation is a (Laurent) polynomial with prescribed monomial support.
We denote by
$\mathcal{A}_{\bullet } = (\mathcal{A}_1, \ldots , \mathcal{A}_n)$ the system's prescribed support sets, $\mathcal{A}_i\subset \ZZ^n,$ and abuse notation by writing $\mathcal{A}_\bullet = \# \mathcal{A}_1 + \ldots + \# \mathcal{A}_n $ for the total number of monomials in the system.
With some mild conditions on $\mathcal{A}_\bullet $, we may consider the branched cover
\begin{align*}
\pi_{\mathcal{A}_\bullet } : \, &X_{\mathcal{A}_\bullet} \to \CC^{\mathcal{A}_\bullet }, \quad (p,x) \mapsto p, \\
&X_{\mathcal{A}_\bullet } = 
\left\{ (p, x) \in \CC^{\mathcal{A}_\bullet } \times  (\CC^*)^n \mid \sum_{\alpha \in \mathcal{A}_i } p_{i, \alpha } \, x^\alpha = 0, \, \, 1\le i \le n \right\} .
\end{align*}
The Bernstein-Kushnirenko (BK) theorem~\cite{bernstein,kushnirenko} writes the degree of $\pi_{\mathcal{A}_\bullet }$ as the mixed volume of the $n$ equations' Newton polytopes,
\begin{equation}
\deg (\pi_{\mathcal{A}_\bullet } ) = \operatorname{MV} (\mathcal{A}_\bullet ) := 
\operatorname{MV} (\operatorname{conv} (\mathcal{A}_1), \ldots , \operatorname{conv} (\mathcal{A}_n)).
\end{equation}
Correspondingly, the polyhedral homotopy method, introduced by Huber and Sturmfels~\cite{huber}, involves tracking $\operatorname{MV} (\mathcal{A}_\bullet)$ solution paths.
This method is optimal in the sense that all paths lead to a solution for ``most" systems $p\in  \CC^{\mathcal{A}_\bullet }$.
However, there are two major caveats:
\begin{enumerate}
    \item[(1)] In many applications, we are led naturally to consider a branched cover~\eqref{eq:branched-cover} which is not of the form $\pi_{\mathcal{A}_\bullet}$. 
    \item[(2)] Even if we are interested in $\pi_{\mathcal{A}_\bullet}$, it may be possible to track fewer than $\operatorname{MV} (\mathcal{A}_\bullet)$-many paths by exploiting symmetry or decomposability into branched covers of lower degree.
\end{enumerate}
In both of these scenarios, the BK theorem typically overestimates both the Galois width and the number of homotopy paths needed.
Similar caveats apply to various other root-counting technologies used to construct start systems for homotopy continuation: namely, extensions of the BK theorem via methods based on Khovanskii bases~\cite{elise,borovik2024khovanskii} or tropical geometry~\cite{helminck2022generic,leykin2019beyond}, and most of the classical approaches cataloged in~\cite[Ch.~8]{sw-book}.
In contrast, the \emph{parameter homotopy} paradigm of ~\cite[Ch.~7]{sw-book} is well-adapted to the setting of this paper, and $\botn (\bullet )$ may be viewed as a general complexity measure for this approach.

Caveats (1) and (2) above are well-known to applied algebraic geometers.
For example, the maximum-likelihood estimation problems of~\Cref{subsec:astat} provide one notable illustration of Caveat (1), as shown by the problems in~\cite[Table 1]{rank-mle} where BK overcounts by multiple orders of magnitude. 
Caveat (2) was addressed in work of Brysiewicz et al.~\cite{brysiewicz}, which showed how polyhedral homotopy could be made more efficient by exploiting special structure in the support sets $\mathcal{A}_\bullet.$
Prior to this work, Esterov~\cite[Theorem 6]{esterov} characterized the supports for which $\Gal( \pi_{\mathcal{A}_\bullet})$ is \emph{full-symmetric}, ie.~the symmetric group acting on $\operatorname{MV} (\mathcal{A}_\bullet) $ letters.
We recall a reformulation of this result in the terminology of~\cite{brysiewicz,sottile2021galois}.
The support sets $\mathcal{A}_1, \ldots , \mathcal{A}_n \subset \ZZ^n$ are said to be \emph{strictly lacunary} if they span a proper sublattice of
$\ZZ^n$ of index less than $\deg (\pi_{\mathcal{A}_\bullet })$.
We also say $\mathcal{A}_\bullet$ is \emph{triangular} if  there exists a proper subset of supports $\mathcal{A}_I = \{ \mathcal{A}_i \mid i\in I \}$, $|I| = k<n,$ spanning a sublattice in $\ZZ^n$ of rank $k$; if this holds, then up to unimodular transformation, we may assume then that $\mathcal{A}_I \subset \ZZ^k,$ and say that $\mathcal{A}_\bullet$ is \emph{strictly triangular} if $1 < \operatorname{MV} ( \mathcal{A}_I) < \operatorname{MV} ( \mathcal{A}_\bullet ).$

\begin{theorem}\label{thm:esterov}
(Esterov,~\cite{esterov})
If $\operatorname{MV} (\mathcal{A}_\bullet ) > 0$, then
$\Gal (\pi_{\mathcal{A}})$ is full-symmetric if and only if $\mathcal{A}_\bullet$ is neither strictly triangular nor strictly lacunary.
\end{theorem}

\begin{example}
Let $\mathcal{A}_i = \{ 0, 2 e_i \},$ where $0, e_i\in \ZZ^n$ are the zero and $i$-th standard basis vector, for $i=1,\ldots , n.$
This encodes the sparse system
\begin{equation}\label{eq:sparse-guy}
p_{1,1} x_1^2 + p_{1,2} =  \ldots = p_{n,1} x_n^2 + p_{n,2} = 0.
\end{equation}
$\operatorname{conv} (\mathcal{A}_i)$ is a segment, and $\operatorname{MV} ( \mathcal{A}_\bullet )$ is the volume of the zonotope 
\begin{equation}
\operatorname{conv} (\mathcal{A}_1) + \ldots + \operatorname{conv} (\mathcal{A}_n) = [0,2]^{\times n},
\end{equation}
as expected since $\deg (\pi_{\mathcal{A}_\bullet })=2^n.$
In this case, the supports $\mathcal{A}_\bullet $ are both strictly lacunary and strictly triangular. 
The function field $\CC ( X_{\mathcal{A}_\bullet })$ is a Galois extension of $\CC( \CC^{\mathcal{A}_\bullet})$, and $\Gal (\pi_{\mathcal{A}_\bullet })\cong \left( \ZZ / \ZZ_2 \right)^{\times n}$.
For this example, the gap between Galois width and degree is exponential in $n$:
\[
\botn ( \pi_{\mathcal{A}_\bullet}) = 2 < 2^n = \deg (\pi_{\mathcal{A}}).
\]
\end{example}

The two conditions appearing in~\Cref{thm:esterov} reflect two birational invariants of a general branched cover related to its Galois width: namely, the group of deck transformations and decomposability.
\begin{definition}\label{def:deck-decomposable}
Let $\pi : X \dashrightarrow Z$ be a branched cover.
A rational map $\Psi : X \dashrightarrow X$ is said to be a \emph{deck transformation} if $\pi ( \Psi (x)) = \pi (x)$ for all $x$ in a common domain of definition. 
If $\pi_1 : X \dashrightarrow Y $ and $\pi_2 : Y \dashrightarrow Z$ are branched covers of degree strictly less than $\deg (\pi) $ with $\pi = \pi_2 \circ \pi_1$ on a common domain of definition, we say that $\pi $ is decomposable, and say that the $\pi_i$ give a decomposition of $\pi .$
\end{definition}
The deck transformations form a group $\Deck (\pi)$ with respect to composition; if this group is nontrivial, then 
\[
X \dashrightarrow X / \Deck (\pi) \dashrightarrow Z
\]
is a decomposition unless $\Gal (\pi )$ is a cyclic group of prime order, in which case also $\Gal (\pi ) \cong \Deck (\pi)$.
When $\Gal (\pi ) \cong \Deck (\pi),$ we say that $\pi $ is a \emph{Galois cover}.
Many previous works have utilized either deck transformations or decompositions, at least implicitly, to speed up numerical homotopy continuation methods.
One may consult~\cite[\S 7.6]{sw-book} for a discussion in the context of parameter homotopy.
Recent articles such as~\cite{brysiewicz,gm-vision,rec-symmetry,amendola2016solving} provide further applications.

Much like~\Cref{prop:bottleneck-normal-recursion}, there is a clean splitting principle that relates the Galois widths of a branched cover and its various decompositions.
\begin{prop}\label{prop:decomp-width}
If $\pi = \pi_2 \circ \pi_1$ is a decomposition, then
\[
\botn \left( 
\pi  
\right)
=
\max \left(
\botn \left( 
\pi_1  \right), \,
\botn \left( 
\pi_2 
\right)
\right).
\]
\end{prop}
\begin{proof}
First, note that $\Gal (\pi_2)$ is a homomorphic image of $\Gal (\pi )$ arising from the action of $\Gal (\pi )$ on fibers $\pi_2^{-1} (z)$. 
Thus,~\Cref{prop:bottleneck-normal-recursion} implies $\botn\left( \Gal (\pi_2) \right) \le \botn\left( \Gal (\pi) \right)$.
Similarly, fixing $y\in \pi_2^{-1} (z),$ $\Gal (\pi_1)$ is a homomorphic image of the subgroup of $\Gal (\pi )$ that stabilizes $\pi_1^{-1} (y)$ setwise. Thus, applying now both~\Cref{prop:bottleneck-normal-recursion} and~\Cref{thm:monotonicity}, we obtain that $\botn \left( \Gal (\pi_1 )\right) \le \botn \left( \Gal ( \pi ) \right).$
Finally, since $\Gal (\pi )$ is a subgroup of the wreath product 
$
\Gal (\pi_1) \wr \Gal (\pi_2)
=
\Gal (\pi_1)^{\deg (\pi_2)} \rtimes  \Gal (\pi_2),
$
we may use~\Cref{prop:bottleneck-normal-recursion} and~\Cref{thm:monotonicity} to deduce that $$\botn (\Gal (\pi )) \le \max\left(  \botn \left(\Gal (\pi_1), \Gal (\pi_2) \right) \right).$$ 
\end{proof}

Now, briefly returning to the setting of sparse systems, we note a corollary of Esterov's~\Cref{thm:esterov} and~\Cref{prop:decomp-width}.
\begin{cor}
The equivalent conditions of~\Cref{thm:esterov} are satisfied if and only if  $\botn \left( \pi_{\mathcal{A}_\bullet}\right) = \operatorname{MV} (\mathcal{A}_\bullet)$. 
\end{cor}
This follows since if $\mathcal{A}_\bullet$ is either strictly lacunary or strictly triangular, then $\pi_{\mathcal{A}_\bullet}$ is decomposable, as observed in~\cite{brysiewicz}.
Galois groups in the decomposable cases are not known, despite continued attention and partial results~\cite{esterov2}.
Determining Galois widths may be more manageable.

Our study of Galois width falls in line with a general philosophy \emph{that intrinsic complexity of solving algebraic equations is best measured through birational invariants of the associated branched covers.}
Degree, Galois width, decomposability, and deck transformations may thus all be viewed as intrinsic measures of complexity for parametric systems of polynomial equations, whereas various bounds on the degree like BK generally cannot be.
This can be contrasted with several interesting, yet fundamentally less intrinsic, studies of when the BK bound is tight (eg.~\cite{tianran,ed-bk-exact,lindberg2023algebraic}.)

From a practical point of view, many of the problems presented in~\Cref{sec:examples}, decompose into solving one very costly ``bottleneck" subproblem and other, smaller, subproblems, and it is tempting to conclude that $\botn (\bullet )$ measures not only the complexity of optimal algebraic algorithms as in~\Cref{sec:bottleneck}, but also an optimal parameter homotopy.
This tempting conclusion, however, should be met with some resistance.

The proof of~\Cref{thm:complexity} can be adapted into a complete procedure for optimally decomposing a Galois cover.
Unfortunately, this procedure is impractical for all but the simplest examples, as it relies on finding irreducible components of iterated fiber products.
Given a branched cover $\pi :X \dashrightarrow Z$ and an integer $k$, we recall the $k$-th fiber product 
\begin{equation}\label{eq:fiber-product}
X_Z^{k} = \{ 
(x_1, \ldots , x_k,z) \in X^{\times k } \times Z \mid \pi (x_i ) = z, \, 1\le i \le k \} .
\end{equation}
(Note: the notation in~\Cref{eq:fiber-product} assumes that $\pi $ is clear from the context.)
When $1\le k \le \deg (\pi ),$ there is a nonempty variety 
\begin{equation}\label{eq:galclo}
    X_Z^{(k)} := X_Z^k \setminus \Delta 
\end{equation} where $\Delta $ is the ``big diagonal" containing all points where $x_i=x_j$ for some $i\ne j.$
The variety $X_Z^{(k)}$ is irreducible if and only if $\Gal (\pi )$ acts transitively on ordered $k$-tuples of solutions $(x_1, \ldots , x_k).$
When $k=\deg (\pi ),$ each top-dimensional irreducible component of $X_Z^{(k)}\setminus \Delta $ is a variety $\hat{X}$ whose function field $\CC (\hat{X})$ is a Galois closure for the extension $\CC (X) / \CC (Z)$. 

The Galois covers $\hat{\pi} :\hat{X} \to Z$ are quite unlike the branched covers one usually meets in practice. 
Note, for example, that the ability to sample a generic point on $\hat{X}$ is the same as computing an entire fiber of $\pi ,$ which would obviate the need for any solving.
Although the argument in~\Cref{thm:complexity} shows that decomposing $\hat{\pi }$ along a chain of subcovers is theoretically optimal, an ad-hoc approach analyzing the original branched cover $\pi$ is far more likely to yield practical results.
On the other hand, decomposing $\pi $ itself may \emph{not} be optimal.
We close with an example from enumerative geometry which reflects some of these difficulties.

\begin{example}\label{ex:cubic-surfaces}
We consider the problem of computing determinantal representations of cubic surfaces.
Given a homogeneous polynomial of degree $3,$ $\gl{f} \in \gl{\CC[x_0, \ldots , x_3]_{3}},$ 
can we find $\gr{M_0, \ldots ,M_3 \in \CC^{3\times 3}}$ with
\[
\gl{f(x_0, x_1, x_2, x_3)} \, \propto \, \det \left( 
\gl{x_0} \gr{M_0} + \gl{x_1} \gr{M_1} + \gl{x_2} \gr{M_2} + \gl{x_3} \gr{M_3}
\right)?
\]
If $f$ is generic, then up to the action of $\gr{\operatorname{GL}_3^{\times 2}} $ defined by $\gr{(H_1, H_2)} \cdot \gr{M_i} = \gr{H_1 M_i H_2^{-1}},$ we may restrict to determinantal representations of the form 
\[
\gr{M_0} = \gr{I},
\,  
\gr{M_1} = \gr{\operatorname{diag} (x_1,x_2,x_3)} ,
\,
\gr{M_2} 
= \gr{\begin{bmatrix}
x_4 & x_5 & x_6 \\
x_5 & x_7 & x_8\\
x_9 & x_8 & x_{10}  
\end{bmatrix}
},
\, 
\gr{M_3} 
= \gr{\begin{bmatrix}
x_{11} & x_{12} & x_{13}\\
x_{14} & x_{15} & x_{16}\\
x_{17} & x_{18} & x_{19}
\end{bmatrix}
}.
\]
This determines a branched cover,
\begin{align*}
\pi : X \to \PP^{19}, 
\quad 
\left( (M_0, M_1, M_2, M_3), f \right) \mapsto f.
\end{align*}
It is known classically that a smooth cubic surface $f$ has exactly $72$ distinct orbits of determinantal representations under the $\gr{\operatorname{GL}_3^{\times 2}}$-action.
In our chosen parametrization, we have $\deg (\pi ) =1728 = 24 \cdot 72.$
The deck transformation group $\Deck (\pi )$ has order $48$; it is generated by the map sending $M_i\mapsto M_i^T$ and a subgroup of order $24$ which preserves $\gr{\operatorname{GL}_3^{\times 2}} $-orbits.
Thus, we have a decomposition $\pi = \pi_2 \circ \pi_1$, where $\pi_1 : X \dashrightarrow X / \Deck (\pi )$ is a Galois cover with $\botn \left( \pi_1\right) = 4$ and $\pi_2: X / \Deck (f) \dashrightarrow \PP^{19}$ with $\deg (\pi_2) =36.$
As it turns out, $\pi_1$ is not decomposable, so we might reasonably guess that $\botn \left(\pi \right) = 
\botn \left(\pi_2\right) = 36$.
But, in fact,
\[
\botn \left(\pi\right) = 
\botn \left(\pi_2\right) =
27,
\]
due to the fact that the determinantal representations of $f$ are rational functions of the $27$ lines on $f$ (cf.~\cite{dolgachev,buckley}.)
With respect to the computational model in~\Cref{subsec:model}, using these rational functions of the $27$ lines is algebraically optimal for computing these determinantal representations.
\end{example}

\section{Applications}\label{sec:examples}

\subsection{Metric Algebraic Geometry}\label{subsec:mag}
\begin{figure}[h]
    \centering
\includegraphics[width=0.75\linewidth]{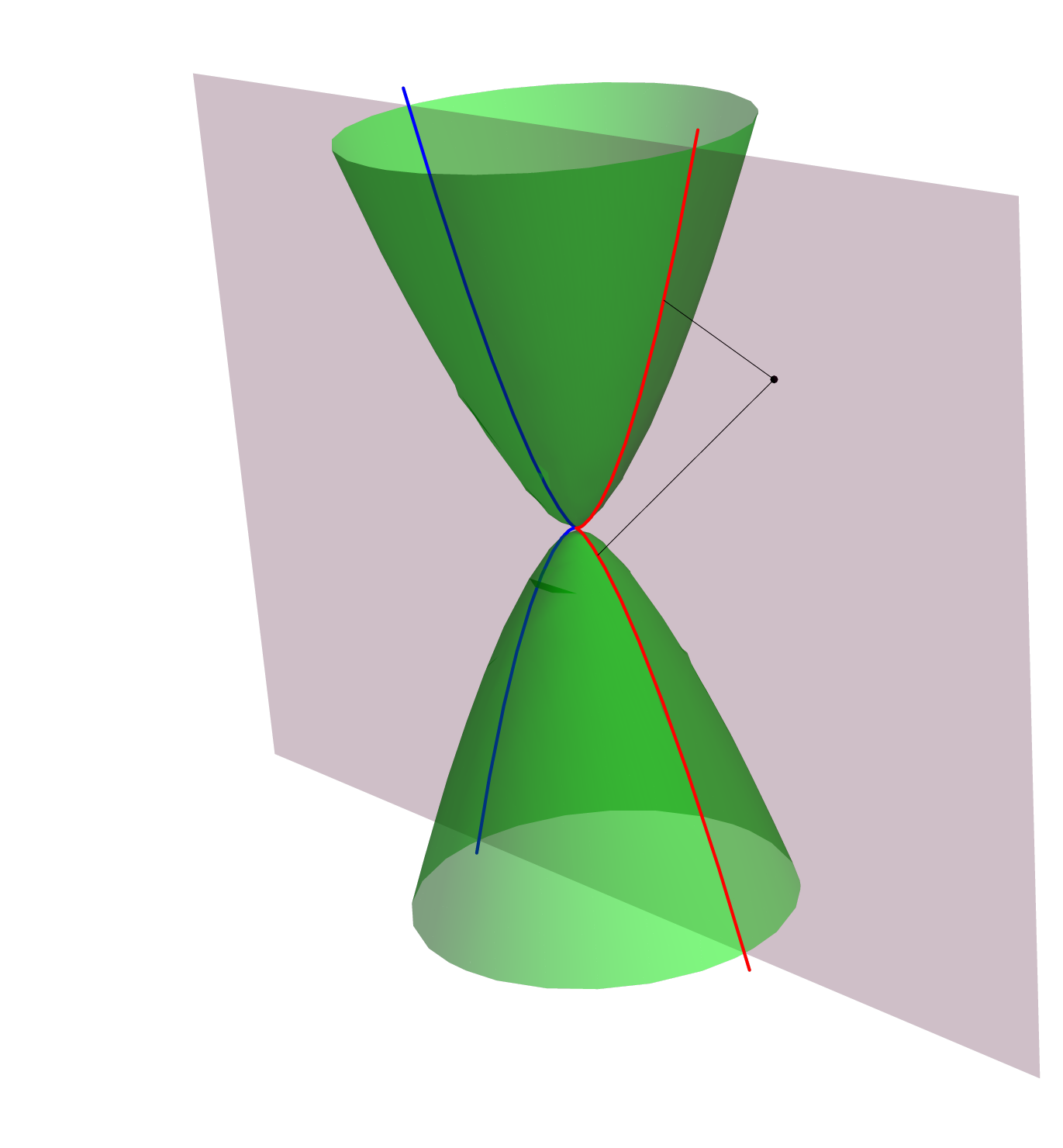}
    \caption{Minimizing the distance from a point in $\RR^3$ to the surface of~\Cref{ex:solid-rev}.}
    \label{fig:cusps}
\end{figure}
The emerging field of \emph{metric algebraic geometry}~\cite{bks} \cite{weinstein2021metric} studies interactions between algebraic varieties and metric structures on spaces they inhabit.
One of the foundational problems in this field is solving the \emph{nearest point problem for a real algebraic variety.}
A general measure of algebraic complexity for this problem is the \emph{Euclidean Distance (ED) degree}.
The systematic study of the ED degree was initiated in~\cite{dhost}, and the subject has received significant attention since.
As we will illustrate, the Galois width captures more refined information about the nearest-point problem and its variants.
For certain, highly-structured problems, we will see how the Galois width
characterizes the difficulty of solving the nearest-point problem more accurately than the ED degree.

Let $X\subset \CC^n$ be a closed, irreducible subvariety.
We assume that the set of real points $X_\RR = X \cap \RR^n$ contains a smooth point $x \in X^{\text{sm}}.$
Given a data point $u=(u_1, \ldots , u_n) \in \RR^n,$ the nearest-point problem is
\begin{equation}\label{eq:close-point}
\displaystyle\min_{x\in X_\RR} \|u-x \|^2.
\end{equation}
Here $\| \bullet \|$ is the usual $2$-norm on $\RR^n.$
Note that the minimum is attained since $X_\RR\subset \RR^n$ is closed.
From an algebraic point of view, the minimizer $x$ in~\eqref{eq:close-point} is just one of many critical points.
These points can be studied using the \emph{Euclidean Distance Correspondence} \begin{equation}
    \mathcal{E}_X = \left\{ 
(u, x) \in \CC^n \times X^{\text{sm}} \mid  
(x-u) \in \left(T_x (X)\right)^\perp
    \right\}
\end{equation}
and an associated branched cover
\begin{equation}
\pi_{\mathcal{E}_X} : \mathcal {E}_X \to \CC^n , 
\quad (u,x ) \mapsto u,
\end{equation}
whose fibers $\pi_{\mathcal{E}_X}^{-1}(u)$ are the complex-valued critical points over $u.$
The ED degree $\ED (X)$, then, is simply the number of points in the fiber for generic $u\in \CC^n.$ 
When $u\in \RR^n,$ the real points in $\pi_{\mathcal{E}_X}^{-1} (u)$ provide candidates for the global minimizer of~\eqref{eq:close-point}.
An additional candidate that must be checked is the closest point to $u$ on the singular locus $X_\RR^{\text{sing}}$.

Noting the general inequality
\begin{equation}
 \botn \left(\pi_{\mathcal{E}_X}\right)\le    \ED (X),
\end{equation}
we now consider an example where this inequality is strict.

\begin{example}\label{ex:solid-rev}
Consider the irreducible surface
\begin{equation}
X = \{ x \in \CC^3 \mid f(x) = 0\},
\quad 
f(x_1,x_2,x_3) = x_3^4  - (x_1^2+x_2^2)^3.
\end{equation}
Given a data point $u=(u_1,u_2,u_3)\in \RR^3$, we may may minimize the squared Euclidean distance function
\begin{equation}\label{eq:npp-surf-1}
d (x ; u ) = (x_1-u_1)^2 + (x_2-u_2)^2 + (x_3-u_3)^2 .
\end{equation}
over $x\in X_\RR$ by introducing a Lagrange multiplier $\lambda $ and writing down the first-order conditions for optimality,
\begin{equation}\label{eq:foo}
\nabla_{x,\lambda} \left( d + \lambda f \right)
=
\begin{bmatrix}
2 (x_1-u_1) -3 \lambda (x_1^2 + x_2^2)^2 x_1\\
2 (x_2-u_2)  -3 \lambda (x_1^2 + x_2^2)^2 x_2\\
2 (x_3-u_3) + 4 \lambda x_3^3 \\
f(x_1,x_2,x_3)
\end{bmatrix}
=
\begin{bmatrix}
0\\0\\0\\0
\end{bmatrix}.
\end{equation}
Using the first two equations in~\eqref{eq:foo}, we see for generic $u$ that any solution $(x_1,x_2,x_3)$  must lie in the plane
\begin{equation}
x_2 ( x_1 - u_1) - x_1 (x_2 - u_2) = u_2 x_1 - u_1 x_2 = 0.
\end{equation}
Intersecting this plane with $X,$ as shown in~\Cref{fig:cusps}, we obtain a plane curve  which is the union of two touching cuspidal cubics,
\begin{equation}\label{eq:cusps}
\left( x_3^2 - \sqrt{\left( 1 + \left(u_2 / u_1\right)^2\right)} x_1^3 \right)
\left( x_3^2 + \sqrt{\left( 1 + \left(u_2 / u_1\right)^2\right)} x_1^3 \right) = 0.
\end{equation}
The ED cover $\pi : \mathcal{E}_X \to \CC^3$ decomposes as $\pi = \pi_2 \circ \pi_1,$ where
\begin{align*}
&\pi_1 : \mathcal{E}_X \to Y , \quad \pi_2 : Y \to \CC^3, \\
&Y = \{ (s,u_1,u_2,u_3) \in \CC^4 \mid u_1^2 (s^2 -1) -  u_2^2= 0 \}. 
\end{align*}
Note that $\deg (\pi_2)=2.$
The cuspidal cubics have ED degree $4$---this may be verified directly from their usual monomial parametrizations. 
From this observation, it also follows that $\deg (\pi_1 ) = 4.$ Thus,
\[
\botn \left(\Gal (\pi_{\mathcal{E}_X})\right)
\le 
\botn (S_4 \wr S_2)
=
\botn (S_4^2 \rtimes  S_2)
=
\min (4,2)
=4 < 8 = \ED (X).
\]
A curious feature of this example is that  $X_\RR^{\text{sing}}= \{ (0,0,0) \}$ only attains the minimum in~\eqref{eq:close-point} when $u=(0,0,0).$
For either one of the cuspidal cubics in~\eqref{eq:cusps}, its cusp attains the minimum distance precisely when $u$ belongs to a \emph{Voronoi cell} in the opposite half-plane (see~\cite[Fig.~2]{voronoi}.) 
The two opposite Voronoi cells intersect only in the origin.

In summary, although $\ED (X)=8,$ the Galois width reveals that nearest-point problem for $X$ admits a closed-form solution in radicals.
\end{example}

For real varieties with highly-structured Euclidean geometry, like the surface of revolution in~\Cref{ex:solid-rev}, computing the Galois width may provide a critical first clue towards a practical solution to the nearest-point problem.
Our next two examples belong to another such special class of varieties, the \emph{orthogonally-invariant matrix varieties} studied in~\cite{oimv1,oimv2}.

\begin{example}\label{ex:ed-essential}
Consider the affine cone of \emph{essential matrices},
\begin{equation}
X^{\text{es}} = \{ E \in \CC^{3\times 3} \mid E E^T E - (1/2) \operatorname{tr} (EE^T) E =0, \, \, \det E =0\},
\end{equation}
a much-studied object in  computer vision.
Over the reals, this variety has a well-known spectral characterization: a nonzero matrix $E$ belongs to $X_\RR^{\text{es}}$ if it has rank 2 and two of its singular values are equal. 
In fact, given a generic matrix $F\in \RR^{3\times 3},$ one may compute the closest point $E_*\in X_\RR^{\text{es}}$ to $F$ via its singular value decomposition (see eg.~\cite[Theorem 5.9]{ma2004invitation}),
\[
F = U \operatorname{diag} \left(\sigma_1, \sigma_2, \sigma_3\right) V^T
\quad \Rightarrow 
\quad 
E_* = U \operatorname{diag} \left(\displaystyle\frac{\sigma_1+\sigma_2}{2},\displaystyle\frac{\sigma_1+\sigma_2}{2},0\right)  V^T.
\]
Furthermore,~\cite[Example 2.7]{oimv2} implies that all critical points are real:
\begin{align*}
\pi_{\mathcal{E}_{X^{\text{es}}}}^{-1} (F) = \Big \{ &\left(F, \, U \operatorname{diag} \left(\displaystyle\frac{\sigma_1\pm\sigma_2}{2},\displaystyle\frac{\sigma_2\pm\sigma_1}{2},0\right)V^T \right) ,\\ 
&\left(F, \, U\operatorname{diag} \left(\displaystyle\frac{\sigma_1\pm\sigma_3}{2},0 ,\displaystyle\frac{\sigma_3\pm\sigma_1}{2}\right) V^T\right) ,\\
&\left(F, \, U \operatorname{diag} \left(0, \displaystyle\frac{\sigma_2\pm\sigma_3}{2},\displaystyle\frac{\sigma_3\pm\sigma_2}{2}\right)V^T\right) \Big\}.
\end{align*}
Thus $\ED (X^{\text{es}})=6$.
However, $\botn \left( \pi_{\mathcal{E}_{X^{\text{es}}}} \right)=3.$
Unlike in~\Cref{ex:solid-rev}, the ED cover for this example has a nontrivial deck transformation,
\[
\Psi : \mathcal{E}_{X^{\text{es}}} \to \mathcal{E}_{X^{\text{es}}} , \quad (F, E) \mapsto (F, F-E).
\]
Here, again, the Galois width reflects the algebraic difficulty of obtaining a solution---namely, the cost of eigendecomposing a $3\times 3$ matrix.
\end{example}

As the next example shows, the gap between the Galois width and ED degree can be exponentially large.

\begin{example}\label{ex:eckart-young}
Fix integers $m\ge n \ge r \ge 1,$ and let $X_{m,n,r}\subset \RR^{m\times n}$ be the closed subvariety of $m\times n$ matrices whose rank is at most $r.$
There is an ``enhanced" version of the celebrated Eckart-Young theorem for the Frobenius norm (see eg.~\cite[Theorem 2.9]{bks}), which implies that
\[
\gl{\operatorname{ED}} (X_{m,n,r}) = \gr{\binom{n}{r}},
\]
and that all ED-critical points are real and may be obtained via SVD: 
\begin{equation*}
A = 
\displaystyle\sum_{i=1}^n \sigma_i \mathbf{u}_i \mathbf{v}_i^T
\quad \Rightarrow 
\quad 
\pi_{\mathcal{E}_{X_{m,n,r}}}^{-1} (A) = \left\{
\sum_{i\in S } \sigma_i  \mathbf{u}_i \mathbf{v}_i^T
\mid 
S\in \binom{[n]}{r}
\right\}.
\end{equation*}
The following result establishes
\begin{equation}\label{eq:gw-lowrank}
\botn (\pi_{\mathcal{E}_{X_{m,n,r}}}) = n ,
\end{equation}
an exponential improvement over the ED degree: for even $n,$
\[
\botn (\pi_{\mathcal{E}_{X_{m,n,n/2}}}) = n \ll \ED (X_{m,n,n/2}) = \binom{n}{n/2} \approx 4^n / \sqrt{\pi n}.
\]
Once more, the Galois width better reflects the algebraic cost of solving the nearest-point problem (eigendecomposing a $n\times n$ matrix.)

\begin{theorem}\label{thm:gal-eckart}
As a permutation group, $\Gal (\pi_{\mathcal{E}_{X_{m,n,r}}})$ is the symmetric group $S_n$ acting on $\binom{[n]}{r},$ the set of all $r$-element subsets of $[n]= \{ 1, \ldots , n\}.$
\end{theorem}
\begin{proof}
Consider first the case $r=1,$ which we prove by induction on $n.$
The base case  $n=1$ is immediate.
Suppose the result result holds for matrices with fewer than $n$ columns.
Fix $A'$ to be a real matrix of size $(m-1)\times (n-1)$ with distinct singular values. 
Let $\sigma_1$ be a real number greater than the singular values of $A'$, and set
\[
A = \begin{bmatrix}
    \sigma_1 & 0 \\
    0 & A'
\end{bmatrix} \in \CC^{m\times n}.
\]
Our inductive hypothesis provides loops based at $A$ which generate a subgroup $S_{n-1} \times S_1 \subset \Gal (\pi_{\mathcal{E}_{m,n,r}})$.
\Cref{lem:stabilizer-better} gives $\Gal (\pi_{\mathcal{E}_{m,n,r}}) = S_n.$

For $r>1,$ consider the Galois closure (eq.~\eqref{eq:galclo}) for the rank-1 case, 
\begin{equation}
\hat{\mathcal{E}}_{X_{m,n,1}} = \left( \mathcal{E}_{X_{m,n,1}}\right)_{\CC^{m\times n}}^{(n)} .
\end{equation}
The map $\hat{\pi}:\hat{\mathcal{E}}_{X_{m,n,1}} \to \CC^{m\times n}$ factors through $\pi_{\mathcal{E}_{X_{m,n,r}}}$, for instance~via the map that sends an ordered $n$-tuple of rank-1 matrices to the sum of the first $r$ matrices. 
Thus, the monodromy action on $\mathcal{E}_{X_{m,n,r}}$ is induced by that of the group $\Gal (\hat{\pi}),$ which is $S_n$ by the $r=1$ case. 
\end{proof}
\end{example}
We point out that in~\Cref{ex:solid-rev} the number of real critical points for generic real-valued data may vary, whereas in~\Cref{ex:ed-essential,ex:eckart-young} they are always real.
Overall, it is an interesting challenge to determine the Galois widths of nearest-point problems on varieties.
The same is true for numerous variations of this problem found in the literature which arise from weighted nearest-point problems~\cite{ottaviani}, $p$-norm optimization~\cite{kubjas2021algebraic}, or ED degrees with respect to general inner products~\cite{kozhasov2023minimal}.
All of these problems correspond to branched covers, and it is interesting to ask when their Galois/monodromy group are full-symmetric.
In the case of weighted ED degrees, the surface $X$ from~\Cref{ex:solid-rev} again provides the answer ``not always."
To minimize the weighted distance function
\begin{equation}\label{eq:npp-surf-w}
d (x ; u, w ) = w_1(x_1-u_1)^2 + w_2(x_2-u_2)^2 + w_3(x_3-u_3)^2 .
\end{equation}
over $X$ from~\Cref{ex:solid-rev}, solving the critical equations again reduces to the case of two touching cuspidal cubics,
\begin{equation}\label{eq:cusps-w}
x_3^2 = \pm  \sqrt{\left( 1 + \left((w_2u_2) / (w_1u_1)\right)^2\right)} x_1^3 ,
\end{equation}
and once more we find s $8$ critical points of Galois width $4$.
On the other hand, for the case of low-rank matrix approximation, the generic weighted ED jumps dramatically---see eg.~\cite[eq.~3.11]{kozhasov2023minimal}.
It seems plausible in this case that the Galois/monodromy group is not merely symmetric (as in~\Cref{thm:gal-eckart}), but full-symmetric.

\subsection{Algebraic Statistics}\label{subsec:astat}

In this section, we further motivate the Galois width with applications to statistical parameter estimation,  one of the primary topics in the active field of algebraic statistics~\cite{sullivant}.
Our discussion is brief, and partly expository---for the experts, we highlight~\Cref{conj:mle,conj:mom}.

Much like the ED degree, the Maximum-Likelihood (ML) degree~\cite{hosten,catanese} bounds the algebraic complexity of minimizing the log-likelihood
\[
\ell (p; u) = \displaystyle\sum_{i=0}^n u_i \log (p_i),
\]
where $u$ is a given data vector, over points $[p_0 : \cdots : p_n]$ of an irreducible projective variety $ X\subset \PP^n$ which lie in the affine chart $\sum_{i=0}^n p_i = 1$ and satisfy $p_i \ne 0$ for all $i.$
To construct the associated branched cover, consider first the vanishing ideal $\mathcal{I}_X = \langle g_1, \ldots , g_k \rangle $ and the rational map
\begin{align*}
    \Psi_X : X \setminus \left( X^{\text{sing}} \cup \mathcal{H}_n \right) \times \PP^{k} &\to \PP^n \\
    \left([p_0 : \cdots : p_n], [\lambda_0 : \cdots : \lambda_{k+1}]\right) &\mapsto 
    \left[ \lambda_0 + \sum_{i=1}^k \lambda_i p_i \displaystyle\frac{d g_i}{d p_j} (p; u) \mid 0\le j \le n \right] ,
\end{align*}
where $\mathcal{H}_n = \VV \left(p_0 \cdots p_n \cdot \left( \sum p_i \right) \right)$ is the union of the $n+2$ hyperplanes.
The image of the graph of $\Psi_X$ under the coordinate projection
\begin{align*}
\operatorname{Gr} (\Psi_X ) \to \PP^n \times \PP^n , \quad 
(p, \lambda , u) \mapsto (u, p)
\end{align*}
is known as the \emph{likelihood correspondence} $\mathcal{L}_X$. When the projection
\begin{align*}
\pi_{\mathcal{L}_X} : \mathcal{L}_X &\to \PP^n, \quad 
(u, p) \mapsto u
\end{align*}
is a branched cover, its degree, $\ML (X) := \deg (\pi_{\mathcal{L}_X} )$, is known as the \emph{ML degree} of $X.$
For generic $u,$ this number counts the complex-valued critical points $p \in X \setminus \left( X^{\text{sing}} \cup \mathcal{H}_n \right)$ for the log-likelihood $\ell (p ; u )$.
\begin{example}\label{ex:ml-linear}
If $V \subset \PP^n$ is a generic linear space of codimension $r,$ then
\begin{equation}
    \ML (V) = \binom{n}{r},
\end{equation}
as shown in~\cite[Theorem 13]{catanese}, which states furthermore that for real data all critical points are real.
Although the number of (real) critical points for this problem and low-rank matrix approximation from~\Cref{ex:eckart-young} coincide, 
Galois widths imply that this problem is much harder.
For example, if we take $(n,r)=(4,2),$ low-rank approximation has Galois width $4 < 6 = \binom{4}{2}.$ In contrast, this problem for generic $V\in \operatorname{Gr} ( \PP^2 , \PP^4)$ has Galois width $6$ and cannot be solved in radicals.
We may take
\begin{align*}
u = \begin{bmatrix}
6&
4&
5&
3&
12
\end{bmatrix}
, \quad 
&V = \VV (
\ell_1 (p) , \ell_2 (p)
), \quad \text{where}\\
\ell_1 (p) = p_1+p_2+2p_3-5p_4-2p_0,  \quad 
&\ell_2 (p) = 3p_1+2p_2-p_3-3p_4+4p_0.
\end{align*}
Solutions to this ML problem are the points on $V \setminus \mathcal{H}_4$ where the matrix
\[
\begin{bmatrix}
u_0 & u_1 & u_2 & u_3 & u_4 \\
p_0 & p_1 & p_2 & p_3 & p_4 \\
p_0 \cdot \displaystyle\frac{\partial \ell_1}{\partial p_0} & p_1 \cdot \displaystyle\frac{\partial \ell_1}{\partial p_1} &
p_2 \cdot \displaystyle\frac{\partial \ell_1}{\partial p_2} &
p_3 \cdot  \displaystyle\frac{\partial \ell_1}{\partial p_3} &
p_4 \cdot  \displaystyle\frac{\partial \ell_1}{\partial p_4}\\
p_0 \cdot \displaystyle\frac{\partial \ell_2}{\partial p_0} & p_1 \cdot \displaystyle\frac{\partial \ell_2}{\partial p_1} &
p_2 \cdot  \displaystyle\frac{\partial \ell_2}{\partial p_2} &
p_3 \cdot  \displaystyle\frac{\partial \ell_2}{\partial p_3} &
p_4 \cdot  \displaystyle\frac{\partial \ell_2}{\partial p_4}
\end{bmatrix}
\]
drops rank.
With a little help from Macaulay2~\cite{M2}, we may obtain a homogenized primitive element
\begin{align*}
f = &1\,126\,944\,p_{0}^{6}+4\,127\,767\,p_{0}^{5}p_{1}-2\,310\,180\,p_{0}^{4}p_{1}^{2}-2\,105\,058\,p_{0}^{3}p_{1}^{3}\\
&+434\,717\,p_{0}^{2}p_{1}^{4}+170\,397\,p_{0}p_{1}^{5}+5\,670\,p_{1}^{6},
\end{align*}
and check that $\Gal (f) = S_6.$
This follows after
reducing $f$ modulo the primes $29$ and $141$ and factoring, thereby showing that $\Gal (f)$ contains Frobenius elements which are, respectively, a $5$-cycle and a transposition.
\end{example}
Although we expect the Galois groups for maximum-likelihood in many cases to be full-symmetric like in~\Cref{ex:ml-linear}, this is not always the case.
\begin{example}\label{ex:rank-mle}
Consider the projectivization $Y_{m,n,r}=\PP (X_{m,n,r})$ of the variety of $m\times n$ rank $\le r$ matrices from~\Cref{ex:eckart-young}.
As before, we assume $r\le n \le m.$
Using numerical homotopy continuation methods, the values $\ML (Y_{m,n,r})$ for $m\le 6$ were completely tabulated in~\cite{rank-mle}.
An analogous table produced ML degrees for rank-constrained symmetric matrices,
\begin{equation}
Z_{n,r} = \left\{ p \in \PP^{(n+2)(n-1)/2} \mid \operatorname{rank} 
\begin{bmatrix}
    2 p_{1,1} & p_{1,2} & \cdots & p_{1,n} \\
    p_{1,2} & 2 p_{2,2} & \cdots & p_{2,n}  \\
    \vdots & \vdots & \ddots & \vdots \\
    p_{1,n} & p_{2,n} & \cdots & p_{n,n}
\end{bmatrix}
\le r 
\right\},
\end{equation}
for $1\le r \le n\le 6,$ except the missing cases $\ML (Z_{6,3}) = \ML (Z_{6,4}) = 68774$ announced in~\cite[Figure 15]{ugen}.
For the particular case of $Z_{3,2}$, it was noticed in~\cite[Proposition 1]{rank-mle} that the Galois group was $S_4$ acting on $6=\binom{4}{2}$ solutions corresponding to pairs $1\le i < j \le 4.$ 
As a consequence, the maximum-likelihood estimate in this case is solvable in radicals.
This fact can be deduced immediately from the existence of an order-2 deck transformation for the branched cover $\pi_{\mathcal{L}_{Z_{3,2}}}$. If we set
\[
P = \left[\!\begin{array}{ccc}
     2\,p_{1,1}&p_{1,2}&p_{1,3}\\
     p_{1,2}&2\,p_{2,2}&p_{2,3}\\
     p_{1,3}&p_{2,3}&2\,p_{3,3}
     \end{array}\!\right],
     \quad 
     U = \left[\!\begin{array}{ccc}
     2\,u_{1,1}&u_{1,2}&u_{1,3}\\
     u_{1,2}&2\,u_{2,2}&u_{2,3}\\
     u_{1,3}&u_{2,3}&2\,u_{3,3}
     \end{array}\!\right],
\]
\[
D (U) = 
\operatorname{diag} \left(2u_{1,1} + u_{1,2}+ u_{1,3},
u_{2,1} + 2u_{2,2}+ u_{2,3},
u_{3,1} + u_{3,2}+ 2u_{3,3}
\right),
\]
then this deck transformation may be given explicitly as
\begin{align*}
\Psi : \mathcal{L}_{Z_{4,2}} &\dashrightarrow 
\mathcal{L}_{Z_{4,2}} \\
(U, P) &\mapsto 
\left( U , \, \,  \displaystyle\frac{32}{(\operatorname{tr} (U))^3}  \cdot D(U) \cdot 
\begin{bmatrix}
\displaystyle\frac{u_{1,1}}{p_{1,1}} & 
\displaystyle\frac{u_{1,2}}{p_{1,2}} & 
\displaystyle\frac{u_{1,3}}{p_{1,3}} \\[1em]
\displaystyle\frac{u_{1,2}}{p_{1,2}} & 
\displaystyle\frac{u_{2,2}}{p_{2,2}} & 
\displaystyle\frac{u_{2,3}}{p_{2,3}} \\[1em]
\displaystyle\frac{u_{1,3}}{p_{1,3}} & 
\displaystyle\frac{u_{2,3}}{p_{2,3}} & 
\displaystyle\frac{u_{3,3}}{p_{3,3}}
\end{bmatrix}
\cdot 
D(U)
\right).
\end{align*}
This is a special case of a maximum-likelihood duality result due to Draisma and Rodriguez, which has an analogue for rectangular, non-symmetric matrices (\cite[Theorems 11 and 6, respectively]{MLDUALITY}).
In general, when $n$ is odd and $r = (n+1)/2,$ these duality results specialize to formulas for deck transformations like $\Psi $ given above.

Based on this discussion, we propose the following conjecture.
\end{example}

\begin{conjecture}\label{conj:mle}
For any integers $1\le r \le n \le m,$ we have
\begin{align*}
\botn \left(\pi_{\mathcal{L}_{X_{m,n,r}}}  \right) 
&=
\begin{cases}
\ML (Y_{m,n,r})/2 & \text{for } n \text{ odd, } r=(n+1)/2,\\
\ML (Y_{m,n,r}) & \text{else,}
\end{cases}
\\
\botn \left( \pi_{\mathcal{L}_{Z_{n,r}}} \right) 
&=
\begin{cases}
\ML (Z_{n,r})/2 & \text{for } n \text{ odd, } r=(n+1)/2,\\
\ML (Z_{,r}) & \text{else.}
\end{cases}
\end{align*}
\end{conjecture}
We note that deck transformations in maximum-likelihood may also arise due to label-swapping symmetries---see eg.~\cite[\S 5]{simon}.
We next recall how such symmetries arise in another statistical parameter estimation problem: the method-of-moments for mixture models.
For concreteness, we consider the classical setting, pioneered by Pearson over a century ago~\cite{pearson1894contributions}, of univariate Gaussian mixture models.

For a mixture of $k$ Gaussians $\mathcal{N}(\mu_1, \sigma_1), \, \ldots , \, \mathcal{N} (\mu_k , \sigma_k)$, the moments $\left(M_{i,k}(\mu, \sigma , \lambda )\right)_{i\ge 0}$ are a sequence of polynomials depending on the mean parameters $\mu = (\mu_1, \ldots , \mu_k)$, standard deviations $\sigma = (\sigma_1, \ldots , \sigma_k),$ and mixing coefficients $\lambda = (\lambda_1, \ldots , \lambda_{k-1}).$ 
These polynomials may be computed recursively by the rules
\begin{align}
&M_{i,k} (\mu , \sigma , \lambda ) = \lambda_1 m_i (\mu_1, \sigma_1) + \ldots + (1-\lambda_1 - \ldots -\lambda_{k-1}) m_i (\mu_k , \sigma_k) ,
\nonumber \\
&m_i (\mu , \sigma) = \begin{cases}
1  &i = 0,\\
\mu \, m_{i-1} (\mu , \sigma ) +  \sigma^2 \displaystyle\frac{d}{d\mu } m_{i-1} (\mu , \sigma ) &i>0.
\end{cases}
\label{eq:gauss-moments}
\end{align}
The method of moments hypothesizes that data $x_1, \ldots , x_n \in \RR$ are samples drawn from a mixture of Gaussians with unknown parameters $(\mu , \sigma , \lambda ).$
To recover these parameters, a sequence of \emph{sample moments} is computed,
\[
s_i = (1/n) (x_1^i + \ldots  + x_n^i),
\quad 
1\le i \le e,
\]
and one may expect $M_{i,k} (\mu , \sigma , \lambda ) \approx s_i $ for all $i$ if the sample size $n$ is large enough.
Thus, we are led naturally to consider the map
\begin{align}
\pi_{e,k} : \CC^{3k-1} \to \CC^e, \quad (\mu, \sigma , \lambda ) \mapsto \left( M_i (\mu , \sigma , \lambda ) \mid 1 \le i \le e \right).
\end{align}

Solving the moment equations $s_i - M_{i,k} (\mu, \sigma , \lambda )$ for $i=1,\ldots , e$ may be understood as inverting the map $\pi_{e,k}.$
An algebro-geometric study of these maps was initiated in the work~\cite{moment-varieties}, which also posed a conjecture consisting of three parts (cf.~\cite[Conjecture 3]{moment-varieties}):
\begin{enumerate}
    \item (\emph{Algebraic identifiability}) The map $\pi_{3k-1, k}$ is a branched cover. In other words, $3k-1$ moments from a generic Gaussian mixture model suffice to recover the parameters up to finitely-many possibilities. This was later settled in the affirmative~\cite{mom-identifiability}, where it was also discovered that generalizations to multivariate Gaussians may fail.
    \item (\emph{Rational identifiability}) The map $\pi_{3k, k}$ is generically a degree $k!$ (Galois) branched cover over its image.
    In other words, $3k$ \emph{exact} moments suffice to recover the parameters of the mixture model up to the label-swapping action by $S_k.$
    This remains open, although the same result for $3k+2$ moments is established in~\cite{lindberg2025estimating}, with analogous results when some parameters are known or known to be equal.
    \item (\emph{Degree of algebraic identifiability}) The authors in~\cite{moment-varieties} made a ``wild guess" that $\deg (\pi_{3k,k})= k! \left((2k-1)!!\right)^2$ based on computations for $k=2,3,$ but numerical results for the case $k=4$ in~\cite{amendola2016solving} provide strong evidence that this is incorrect.
    Although a degree formula remains elusive, we may loosely interpret the spirit of their conjecture to mean that the moment equations have no ``hidden symmetries."
\end{enumerate}

 \begin{conjecture}\label{conj:mom}
For all $k\ge 1,$ $\botn \left( \pi_{3k-1, k} \right) = \deg (\pi_{3k-1, k} ) / k! $
\end{conjecture}

In summary, the Galois width provides a novel tool for studying the algebraic complexity of statistical parameter estimation problems. We feel optimistic that it will find further applications in the study of identifiability questions and in concrete parameter estimation scenarios.

\subsection{Algebraic Vision}\label{subsec:avis}

The computational model studied in~\Cref{sec:bottleneck} originates from work of Hartley, Nist\'{e}r, and Stew\'{e}nius~\cite{hartley}, who pioneered the application of Galois theory to the study of 3D reconstruction and related problems in computer vision.
More specifically, their paper outlined a framework based on symbolic computation that could be applied to estimation problems with Galois groups that are full-symmetric/alternating.
In many respects, our work is an extension of theirs: we have characterized complexity in our model for any finite group, and can obtain lower bounds on this complexity using more practical numerical heuristics.

On the other hand, there are some notable differences between the setup of~\cite{hartley} and ours; for example, their work treats extracting $n$-th roots in the working field as a unit-cost operation, and possibly additional oracles like the SVD of an arbitrarily-large matrix with real entries.
It would certainly be worthwhile to extend the computational model of~\Cref{sec:bottleneck} to such settings, but this is beyond our current scope.
Rather, we see a discussion of motivating problems from computer vision as a fitting end to this paper.
Although we can only offer a glimpse of these problems and their practical importance, we recommend the recent survey of Kileel and Kohn~\cite{kileel2022snapshot} as one entry point into the expanding field of \emph{algebraic vision}.

Interest in algebraic complexity from the computer vision community stems largely (but not exclusively) from the use of \emph{minimal solvers}~\cite{DBLP:journals/pami/KukelovaBP12,DBLP:conf/cvpr/LarssonAO17} in outlier-robust estimation schemes such as RanSaC (``Random Sampling and Consensus")~\cite{fischler1981random}.
The rough idea behind these methods is this: unknown quantities are estimated multiple times using random subsamples of the data, and these estimates are aggregated so as to filter outliers and produce a more robust estimate.
To minimize the influence of outliers, the subsamples should ideally be as small as possible (``minimal").
Algebraic methods such as Gr\"{o}bner bases are currently the favored approach for constructing minimal solvers. 
The Galois width measures a fundamental limit of what degree of algebraic simplification these solvers can achieve. 

Aside from minimal problems, the algebraic complexity of nonlinear least-squares problems in vision has also attracted significant interest. We revisit the problem of $\ell_2$-optimal triangulation in~\Cref{ex:triangulation} below.

In this paper, we adopt the point of view that estimation problems in vision, minimal or not, may be profitably studied via branched covers (eg.~the \emph{joint camera map} of~\cite{plmp-duff}.)
The Galois width provides a general tool for measuring both the difficulty of these problems and optimality of proposed solutions which is more refined than the degrees of these branched covers.
If the Galois width of a problem is large, then it is truly \emph{hard} in the sense of~\cite{hruby2023four}.
If, however, the Galois width is smaller than the degree, then there is an opportunity for a more efficient algebraic solution.

\begin{example}[Triangulation]\label{ex:triangulation} 
The standard model for the pinhole camera is a $3\times 4$ matrix $P$ of rank $3$~\cite[Ch.~6]{hartley2003multiple}.
We model a world point $X \in \PP^3$ as a $4\times 1$ vector of homogeneous coordinates, so that 
\begin{align*}
 \PP^3 &\dashrightarrow \PP^2 , \quad X \mapsto P \cdot X
\end{align*}
associates each 3D point $X \in \PP^3 \setminus \{ [\ker P]\}$ with a 2D point in $3\times 1$ homogeneous coordinates.
In practice, of course, one must work with affine pixel coordinates in the image, and hence also the rational map
\begin{align*}
\PP^3 &\dashrightarrow \CC^2 , \quad X \mapsto \Pi (P \cdot X), \quad \text{where} \quad  \Pi (x,y,z) = (x/z, y/z). 
\end{align*}
If we have an arrangement of multiple cameras $\mathcal{P}_\bullet = (P_1, \ldots , P_n),$ then one is naturally led to consider the \emph{affine multiview variety}, a variety $\mathcal{M}_{\mathcal{P}_\bullet } \subset \CC^{2n}$ obtained as the Zariski closure of the image of 
\begin{align*}
\PP^3 \dashrightarrow \CC^{2n} , \quad X \mapsto \left(
\Pi (P_1 \cdot X) , \ldots , \Pi (P_n \cdot X) \right).
\end{align*}

In the computer vision literature, the \emph{$\ell_2$-optimal triangulation problem}~\cite{hartley1997triangulation,stewenius2005hard}
is an instance of the 
Euclidean distance minimization problem, already studied in~\Cref{subsec:mag}, for which the following result is known.
\begin{theorem}\label{thm:max-rod-wang} \cite{maxrodwang} 
If $\mathcal{P}_\bullet $ is a generic arrangement of $n\ge 2$ cameras, then 
\begin{equation}\label{eq:cubic}
\ED \left( \mathcal{M}_{\mathcal{P}_\bullet } \right) = \displaystyle\frac{9}{2} n^3 - \displaystyle\frac{21}{2} n^2 + 8n - 4.
\end{equation}
\end{theorem}
We propose a refinement of this result.
\begin{conjecture}\label{conj:multiview}
With notation as in~\Cref{thm:max-rod-wang}, the Euclidean distance cover $\pi : \mathcal{E}_{\mathcal{M}_{\mathcal{P}_\bullet }} \to \CC^{2n}
$ has full-symmetric Galois/monodromy group, and thus its Galois width is thus also given by the formula~\eqref{eq:cubic}.
\end{conjecture}
When $n=2,$~\Cref{conj:multiview} was verified in~\cite{hartley} by direct symbolic computation; using numerical homotopy continuation methods, we have verified it (heuristically) for a few larger values of $n.$
We note that, due to different formulations of the problem, different root counts have been reported in the literature.
For example, when $n=3,$ when the ED degree equals $47,$ but~\cite[Table 2]{chien2022gpu} reports  $94 = 2 \cdot 47$ solutions; nevertheless, the branched cover implicit in their work still has Galois width $47.$

Analogously to~\cite{kozhasov2023minimal} in the case of low-rank matrix approximation, recent work of Rydell et al.~\cite{rydell25weights} shows that certain \emph{weighted} variants of $2$-view triangulation are solvable in radicals.
Although we believe the Galois width may be useful for analyzing similar problems, our framework does not address the question of how the ``best" weights might be chosen.
\end{example}

\begin{example}[Resectioning / Absolute Camera Pose]\label{ex:resectioning} 
The triangulation problem can be understood as fitting the best 3D point to $n$ observed 2D projections under known cameras.
Dually, the \emph{camera resectioning} problem requires fitting an unknown camera to $n$ known 3D-2D point correspondences.
In general, $n\ge 6$ correspondences are needed to solve the problem resectioning. ED degree conjectures for such problems, of a similar flavor to~\Cref{thm:max-rod-wang}, appear in~\cite{connelly2024algebra,duff2024metric}. 

Resectioning can be solved with $n<6$ correspondences if the camera is known to be \emph{calibrated}; that is, when $P$ is given up to scale as an element of the special Euclidean group, $P \propto (R \, \vert \, t ) \in \operatorname{SE}_3 .$
Calibrated resectioning in the minimal case of $n=3$ points is known as the \emph{Perspective-3-Point} problem, or P3P.
Early solutions to this problem date back to Lagrange in 1700s and Grunert (see~\cite{sturm2011historical} for a history), and involve computing the roots of a univariate quartic.
From~\eqref{eq:bot-sn-an}, it follows that the solutions to this problem have Galois width at most $3$, and in~\cite{hartley} it was verified that equality holds.
We note that several solutions to P3P naturally reduce to the case of a cubic equation, without using resolvents (see eg.~\cite{ding2023revisiting}.)

A Galois-theoretic study of camera resectioning with minimal combinations of point and line features was recently carried out in~\cite{gm-vision}.
Letting $p$ and $l$ be nonnegative integers with $p+l=3,$ and $\mathbb{G}_{k,m}$ denote the Grassmannian of $k$-planes in $\mathbb{P}^m,$ then a minimal solution to the resectioning problem with $p$ point-point correspondences and $l$ line-line correspondences amounts to inverting a branched cover of the form
\begin{align*}
\pi_{p,l} : 
\left( \mathbb{P} \right)^{\times p}
\left(\mathbb{G}_{1,3}\right)^{\times l} 
\times \gr{\operatorname{SE}_3} &\dashrightarrow 
\left(\PP^3 \times \gl{\PP^2} \right)^{\times p} \times \left(\mathbb{G}_{1,3} \times \mathbb{G}_{1,2} \right)^{\times l} .
\end{align*}
The main result in~\cite[\S 3]{gm-vision} implies that 
\begin{equation}\label{eq:bot-absolute-pose}
\botn (\pi_{3,0}) = \botn( \pi_{1,2}) = 3, 
\quad 
\botn (\pi_{2,1}) = 2, 
\quad 
\botn (\pi_{0,3}) = 8.
\end{equation}
Practical consequences of this result were recently realized in~\cite{hruby2024efficient}, in which the authors developed new solutions to the minimal cases with $(p,l) \in \{ (2,1), (1,2) \}$. These minimal solvers were vetted in experiments in which they outperformed the previous state-of-the-art for these problems in terms of both accuracy and speed.
Exploiting the gap between Galois width and degree was a key factor that made this possible.
\end{example}

\begin{example}[Reconstruction / Relative Camera Pose]\label{ex:resectioning} 

When both cameras and 3D points are unknown, and only 2D points are given, we are faced with the full \emph{3D reconstruction problem}, which has been extensively studied in both the uncalibrated and the calibrated cases.
In either setting, there is an inherent ambiguity due to choice of coordinates; in other words, solutions are defined up to the action of the projective group $\operatorname{PGL}_4$ or its subgroup consisting of 3D similarity transformations, respectively.
Since knowledge of calibration is often available, and necessary for obtaining metrically-accurate reconstructions, we mostly focus on this case.

In~\cite{plmp-duff}, minimal problems for reconstructing a configuration of points and lines in space were completely classified under the assumptions of calibrated cameras and complete correspondences across multiple images.
A key player in this story was the joint camera map of such a configuration, a branched cover which sends an equivalence class of solutions (configuration + cameras) to multiple 2D projections.
Recent work~\cite{kiehn2025plmp} extends this result to the uncalibrated case, and also suggests novel a novel strategy for checking that the joint camera map is decomposable in the sense of~\Cref{def:deck-decomposable}.

In the calibrated case, the simplest instance of this joint camera map arises for a configuration of five points in space. 
We define
\begin{align}
\pi_{\text{5pp}} : \left(\PP^3\right)^{\times 5} \times \mathcal{C}_{\text{5pp}} &\dashrightarrow  \left(\PP^2 \right)^{\times 10} \\ 
(X_1, \ldots , X_5, P_1 , P_2) &\mapsto \left( P_1 \cdot X_1, \ldots , P_1 \cdot X_5, P_2 \cdot X_1, \ldots , P_2 \cdot X_5\right), \nonumber
\end{align}
where $\mathcal{C}_{\text{5pp}}$ is a variety whose points represent generic pairs of calibrated cameras up to world similarity transformation, specifically
\begin{equation}\label{eq:calcam}
\mathcal{C}_{\text{5pp}} = \{
(P_1 , P_2) \in \left(\operatorname{SE}_3\right)^{\times 2} \mid 
P_1 = \begin{bmatrix} 
I \mid 0
\end{bmatrix}, \, 
P_2 = \begin{bmatrix} 
R \mid t
\end{bmatrix},
\, 
t_3=1
\}.
\end{equation}
Computing fibers of $\pi_{\text{5pp}}$ over generic data $(x_1,\ldots , x_5, y_1, \ldots , y_5)\in \left(\PP^2 \right)^{\times 10}$ is the celebrated \emph{five-point relative pose problem}.
Work on this problem dates back to Kruppa in 1912~\cite{kruppa1912einige}, whose work implies $\deg (\pi_{\text{5pp}}) \le 22$.
This was sharpened much later by Demazure in 1988~\cite{demazure1988deux}, who showed that $\deg (\pi_{\text{5pp}} ) = 20$.
His proof implicitly relied on a decomposition of the joint camera map, $\pi_{\text{5pp}} = \pi_2 \circ \pi_1$, based on a projectivization $\PP (X^{\text{es}}) \subset \PP^8$ of the essential matrix variety already discussed in~\Cref{ex:ed-essential}.
We take
\begin{align}\label{eq:5pp-decomp}
Y_{\textrm{5pp}} = \{ (E , x_1, \ldots , x_5, &y_1, \ldots ,y_5) \in \PP (X^{\text{es}}) \times \left( \PP^2 \right)^{\times 10} \mid y_i^T E x_i = 0 \},\\
\pi_1 : \left(\PP^3\right)^5 \times \mathcal{C}_{\textrm{5pp}} &\dashrightarrow  Y, \nonumber \\ 
(X_1, \ldots , X_5, P_1 , P_2) &\mapsto \left( E(P_1, P_2), P_1 \cdot X_1, \ldots , P_1 \cdot X_5, P_2 \cdot X_1, \ldots , P_2 \cdot X_5\right), \nonumber
\end{align}
where $E(P_1, P_2)\in \PP (X^{\text{es}})$ is a $3\times 3$ matrix whose $(i,j)$-th entry equals 
\begin{equation}
e_{i,j} = (-1)^{i+j} \det \left[
\begin{array}{c}
    P_1 [ \{ 1, 2, 3\} \setminus \{ i \} , \, : ]\\
    \hline
     P_2 [ \{ 1, 2, 3\} \setminus \{ j \} , \, : ]
\end{array}
\right],
\end{equation}
and $\pi_2 : Y \to \left( \PP^2 \right)^2$ is the coordinate projection.
We have $\deg (\pi_1) =10,$ and $\deg (\pi_2)=2,$ implying that $\Gal (\pi_{\text{5pp}}) $ is a subgroup of the wreath product $S_2 \wr S_{10}.$
In~\cite[\S 4]{gm-vision}, the authors obtained strong evidence that $\Gal (\pi_{\text{5pp}}) \cong S_2 \wr S_{10} \cap A_{20} \cong  S_2^9 \rtimes S_{10}$ using numerical homotopy continuation methods. They also proved that the natural map $\Gal (\pi_{\text{5pp}}) \to S_{10}$ is surjective. Thus,
\begin{equation}\label{eq:bot-essential}
\botn (\pi_{\textrm{5pp}}) = 10.
\end{equation}
This agrees with a result of~\cite{hartley}, implying that general methods for solving five-point relative pose, such as Nist\'{e}r's popular solution~\cite{nister2004efficient}, must invariably resort to computing the roots of a degree-$10$ polynomial.

\begin{figure}[h]
    \centering
    \includegraphics[width=0.9\linewidth]{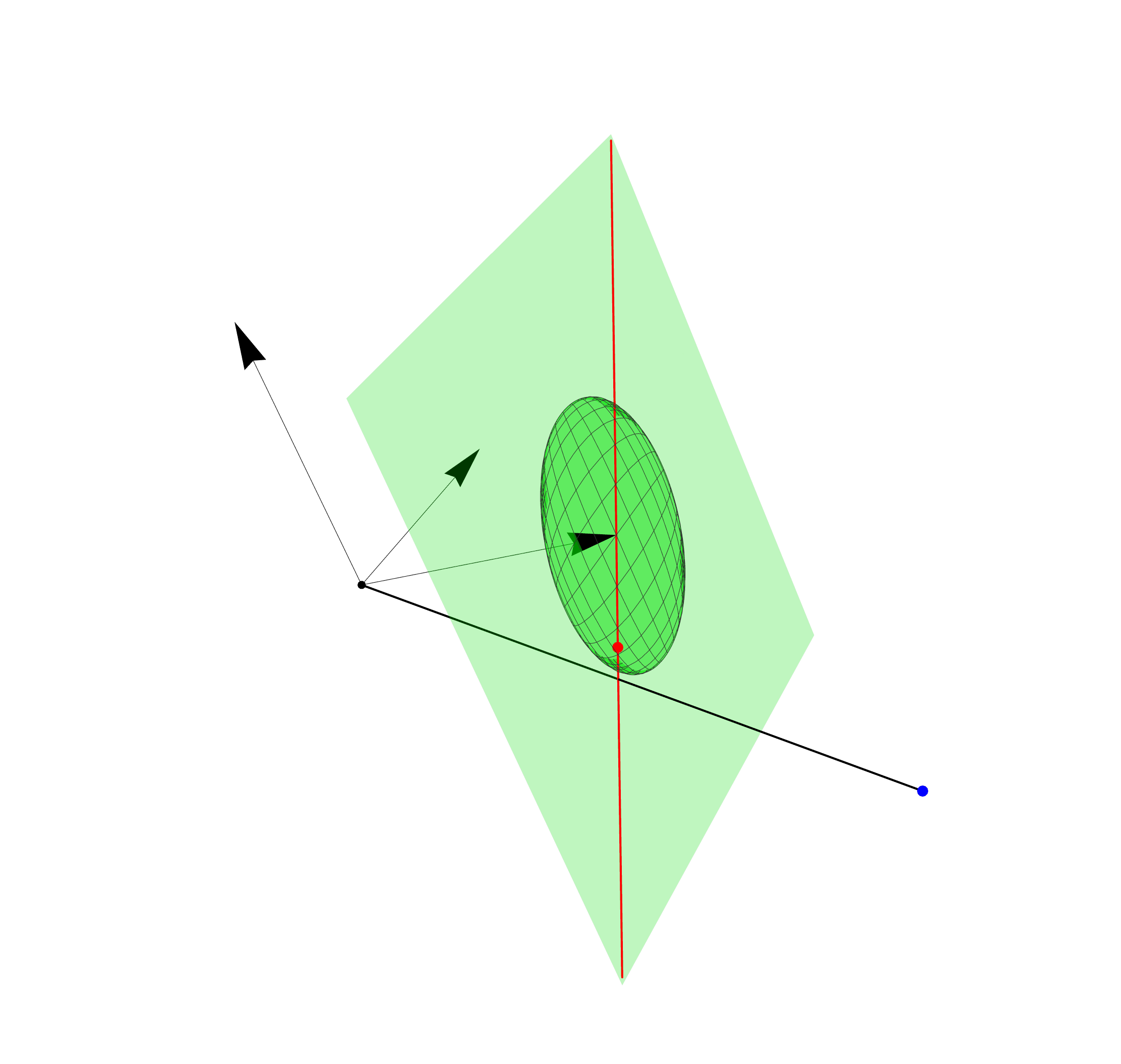}
    \caption{A 3D point (blue) and its 1-dimensional radial projection (red.)}
    \label{fig:radial}
\end{figure}

We conclude with a recent example of a minimal reconstruction problem solved by Hruby et al.~\cite{hruby2023four}, where the gap between the degree of the branched cover and its Galois width is dramatic.
The problem arises from a non-pinhole camera model known as the \emph{radial camera model}~\cite{DBLP:journals/trob/Tsai87,ThirthalaPollefeys,larsson2020calibration,hruby2023four}.
This model addresses the physical phenomenon of lens distortion, a common source of deviations between image measurements and the pinhole model.
\Cref{fig:radial} illustrates a case of ``barrel distortion"; under pinhole projection, straight lines in $\PP^3$ map to straight lines in $\PP^2,$ but distortion causes these lines to bend outward.
To remove the effects of distortion, a number of parametric distortion models have been proposed in the literature.
In contrast, the radial camera model may be seen as a relaxation of the pinhole model, under the assumption that distortion is radially symmetric about some point in the image (the \emph{center of distortion}.)
With this assumption, the straight line through the center of distortion and perspective projection of a 3D point also passes through any point obtained by distorting that projection.
This is the \textcolor{red}{radial line} illustrated in~\Cref{fig:radial}.
The pencil of lines through the center of distortion forms a  $\PP^1,$ and the mapping $\PP^3 \dashrightarrow \PP^1$ from \textcolor{blue}{3D points} to \textcolor{red}{radial lines} is projective-linear, as we now explain.
Consider a pinhole camera $P.$
Without loss of generality, we may assume that the image's center of distortion is the point $e_3 = [0:0:1]\in \PP^2.$ 
For any \textcolor{blue}{3D point} $X\in \PP^3,$ we then have
\begin{equation}
\text{span}_{\PP^2} \left(P \cdot X , e_3\right) =
\text{span} \left( \left[ \begin{array}{c}
     \tilde{P}\\
     0
\end{array} \right] \cdot X , e_3 \right),
\end{equation}
where $\tilde{P} = P[1{:}2,\, :]$ is a $2\times 4$ matrix, representing the \emph{radial camera} $\tilde{P} :\PP^3 \dashrightarrow \PP^1$ associated to $P.$

Given a radial line $\ell \in \PP^1$ and a radial camera $\tilde{P}: \PP^3 \dashrightarrow \PP^1,$ note that the fiber $\tilde{P}^{-1} (\ell)$ is a plane. 
Thus, in order to obtain a nontrivial constraint on a 3D point $X,$ it should be seen by at least $4$ radial cameras $\tilde{P}_1, \ldots , \tilde{P}_4.$
If we assume the original pinhole cameras $P_1, \ldots , P_4$ are calibrated, then
\begin{equation}\label{eq:radial-1}
\tilde{P}_i = \gr{
\begin{bmatrix}
\mathbf{r}_{i1}^T & t_{i,1} \\
\mathbf{r}_{i2}^T & t_{i,2}
\end{bmatrix}
}
\quad 
\text{ where }
\quad 
\mathbf{r}_{i1}^T \mathbf{r}_{i1} =\mathbf{r}_{i2}^T \mathbf{r}_{i2}  =  1,
\, \, 
\gr{\mathbf{r}_{i1}^T}\,\gr{\mathbf{r}_{i2}} = \gr{0}
.
\end{equation}
Up to similarity transformation in $\bl{\RR^3},$ we may assume
\begin{equation}\label{eq:radial-2}
\gr{A_1} = 
 \gr{
\begin{bmatrix}
1 & 0 & 0 & 0\\
0 & 1 & 0 & 0
\end{bmatrix}
},
\quad 
\gr{A_2} = 
 \gr{
\begin{bmatrix}
\mathbf{r}_{21}^T & 0 \\
\mathbf{r}_{22}^T & 1
\end{bmatrix},
}
\end{equation}
and define
\begin{equation}
    \mathcal{C}_{\text{13pp}} = \{ \tilde{P}_1, \ldots , \tilde{P}_4) \in \CC^{2\times 4} \mid \text{\eqref{eq:radial-1} and~\eqref{eq:radial-2} hold} \}.
\end{equation}
Observe that $\dim (\mathcal{C}_{\text{13pp}}) = 13$.
Thus, it is little surprise that the map
\begin{align}
\pi_{\text{rad}} : \left( \PP^3 \right)^{13} \times \mathcal{C}_{\text{13pp}} &\dashrightarrow \left( \PP^1 \right)^{\times 52}\nonumber \\
\quad (X_1, \ldots , X_{13}, \tilde{P}_1,\ldots , \tilde{P}_4) &\mapsto (\tilde{P}_1 \cdot X_1, \ldots , \tilde{P}_4 \cdot X_{13} )
\end{align}
is a branched cover.
What is more surprising is the following.
\begin{theorem}\label{thm:radial}
$\botn (\pi_{\textrm{13pp}}) = 28 < \deg (\pi_{\text{13pp}}) = 3584.$
\end{theorem}
\begin{proof}
In~\cite{hruby2023four}, it was shown that  $\pi_{\text{13pp}}$ has a decomposition of the form
\[
\pi_{\text{13pp}} = \pi_4 \circ \pi_3 \circ \pi_2 \circ \pi_1 ,
\]
where $\deg (\pi_1) = 16,$ $\deg (\pi_2) = 4$, $\deg (\pi_3)=2,$ and $\deg (\pi_1) = 28.$
Now, applying~\Cref{prop:decomp-width} and~\Cref{thm:monotonicity}, we have that
\begin{equation}
\botn (\pi_4) = 28 
\quad 
\Rightarrow 
\quad 
\botn (\pi_{\text{13pp}} ) = 28.
\end{equation}
To see that $\botn (\pi_4) = 28,$ consider the branched cover $\pi' = \pi_3 \circ \pi_2 \circ \pi_1$.
We recall, also from~\cite{hruby2023four}, a description of this map analogous to~\eqref{eq:5pp-decomp};
\begin{align}\label{eq:5pp-decomp}
\pi ' &: \left(\PP^3\right)^5 \times \mathcal{C}_{\textrm{5pp}} \dashrightarrow  Y_{\text{13pp}},\\
Y_{\textrm{13pp}} &= \{ (T , \ell_{1,1}, \ldots , \ell_{13, 4}) \in X^{\text{quad}} \times \left( \PP^1 \right)^{\times 52} \mid T( \ell_{i,1}, \ldots , \ell_{i,4} ) = 0 \}, \nonumber 
\end{align}
where $X^{\text{quad}} \subset \PP (\CC^{2\times 2 \times 2 \times 2})$ is the projective variety of binary tensors $T$ whose entries are given by
\begin{equation}
T_{i,j,k,l} = (-1)^{i+j+k+l} \det \left[
\begin{array}{c}
    \tilde{P}_1 [ \{ 1, 2\} \setminus \{ i \} , \, : ]\\
    \hline
     \tilde{P}_2 [ \{ 1, 2\} \setminus \{ j \} ,\,  : ]\\
     \hline 
         \tilde{P}_3 [ \{ 1, 2\} \setminus \{ k \} , \, : ]\\
    \hline
     \tilde{P}_4 [ \{ 1, 2\} \setminus \{ l \} , \, : ]
\end{array}
\right].
\end{equation}
This is a $13$-dimensional variety of degree $28$, which happens to coincide with the image of the $4\times 4$ principal minor map studied in~\cite{st-lin}.
The map $\pi_4 : Y_{\text{13pp}} \to \left( \PP^1 \right)^{\times 52}$ is given by coordinate projection.
Each $4$-tuple $T(\ell_{i,1}, \ldots , \ell_{i,4})$ gives rise to a linear form in the entries of the tensor $T.$
Considering all such linear forms, we obtain a rational map $\left( \PP^1 \right)^{52} \dashrightarrow \mathbb{G}_{13,15}$.
A computatation in local coordinates shows that this map is dominant.
As such, we may consider the monodromy action of the $\mathbb{G}_{13,15}$ on points in $Y_{\text{13pp}}$ obtained by slicing, which is well-known to be $S_{28}$~\cite[pp.~11-12]{arbarello}.
Thus, by~\Cref{cor:hit}, a generic instance of the $13$-point problem with rational coordinates has Galois group $S_{28}$.
Since $\Gal (\pi_4)\unlhd S_{28}$, we deduce that $\botn (\pi_4) = 28$ as well.
\end{proof}
The decomposition used in the proof of~\Cref{thm:radial} was also used in~\cite{hruby2023four} to devise an efficient homotopy continuation solver for the minimal radial camera reconstruction problem.
At the time of writing, a practical symbolic solution to this problem is still unknown.
Naturally, tracking 28 representative solution paths is more efficient than tracking all 3584.

This final example embodies the main takeaway from this paper: \\
\emph{Don't give up right away on solving a ``hard" system of algebraic equations; first, determine the problem's Galois width, then, plan accordingly.}
\end{example}

\section*{Acknowledgements}

I wish to thank the following individuals for helpful conversations and comments during the preparation of this work: Jon Hauenstein, Kathl\'{e}n Kohn, Kisun Lee, Julia Lindberg, and Jose Israel Rodriguez.

\bibliographystyle{plain}
\bibliography{refs}

\end{document}